\documentclass[a4paper, 11pt]{amsart}
\usepackage{a4wide}
\usepackage{latexsym, fullpage}
\usepackage[left=2cm,top=3cm,right=2cm]{geometry}
\usepackage[utf8]{inputenc}
\usepackage{amsmath}
\usepackage{amsfonts}
\usepackage{amsthm}
\usepackage{amssymb}
\usepackage{amscd}
\usepackage[all]{xy}
\usepackage{fancyhdr}
\usepackage{enumerate}
\usepackage{accents}
\usepackage{setspace}
\usepackage{xcolor}
\usepackage{upgreek}
\usepackage{verbatim}
\usepackage[mathscr]{eucal}
\usepackage{mathrsfs}
\usepackage[numbers]{natbib}
\usepackage[fit]{truncate}
\usepackage{graphicx}
\usepackage{epsfig}
\usepackage{soul}

\usepackage{tabularx,ragged2e,booktabs,caption}
\newcolumntype{C}[1]{>{\Centering}m{#1}}

\numberwithin{equation}{section}

\theoremstyle{plain}
\newtheorem{theorem}{Theorem}[section]
\newtheorem{corollary}[theorem]{Corollary}
\newtheorem{lemma}[theorem]{Lemma}
\newtheorem{proposition}[theorem]{Proposition}

\newtheorem{remark}[theorem]{Remark}

\theoremstyle{definition}
\newtheorem{definition}[theorem]{Definition}
\newtheorem{example}[theorem]{Example}

\theoremstyle{remark}

\newcommand{\R}{{\mathbf R}}

\newcommand{\NN}{\mathbb{N}}

\newcommand\dom{{\mathop{\rm dom}}}

\newcommand{\Gzero}{{\rm (G0)}}
\newcommand{\Gone}{{\rm (G1)}}
\newcommand{\Gtwo}{{\rm (G2)}}
\newcommand{\Gthree}{{\rm (G3)}}
\newcommand{\Gfour}{{\rm (G4)}}
\newcommand{\Gfive}{{\rm (G5)}}
\newcommand{\Gsix}{{\rm (G6)}}
\newcommand{\Gseven}{{\rm (G7)}}
\newcommand{\Geight}{\rm (G8)}
\newcommand{\nul}{{\emptyset}}
\newcommand{\Lip}{\operatornamewithlimits{Lip}}

\newcommand{\yG}{ y_G}

\allowdisplaybreaks

\title{On concavity of the monopolist's problem 
facing consumers with nonlinear price preferences$^*$ 
}
\thanks{$^*$The authors are grateful to Guillaume Carlier, Ivar Ekeland, Alfred Galichon,
 and Xianwen Shi for stimulating conversations,
and to Georg N\" oldeke and Larry Samuelson  for sharing their work with us in preprint form.
The first author's research was supported in part by NSERC grant 217006-08 and -15, 
by a Simons Foundation Fellowship, and by a two-week residency at the Institut Mittag-Leffler.
This project was initiated during the Fall of 2013 when both authors were in residence at the Mathematical Sciences Reseach Institute in Berkeley CA, under a program supported by National Science Foundation Grant No. 0932078 000, and progressed at the Fall 2014 program of the Fields Institute for the Mathematical
Sciences.
\copyright \today}

\author[1]{Robert J. McCann$^\dagger$}\thanks{$^\dagger$Department of Mathematics, University of Toronto, Toronto, Ontario, Canada, M5S 2E4 {\tt mccann@math.toronto.edu}}
\author[2]{Kelvin Shuangjian Zhang$^\ddagger$}\thanks{$^\ddagger$Department of Mathematics, University of Toronto, Toronto, Ontario, Canada, M5S 2E4 {\tt szhang@math.toronto.edu}}
\begin{document}

\begin{abstract}
	A monopolist wishes to maximize her profits by finding an optimal price policy. After she announces a menu of products and prices, each agent $x$ will choose to buy that product $y(x)$ which maximizes his own utility, if positive. 
The principal's profits are the sum of the net earnings produced by each product sold.  
These are determined by the costs
of production and the distribution of products sold, which in turn are based on the distribution of anonymous agents and
the choices they make in response to the principal's price menu.  In this paper, we provide a necessary and sufficient
condition for the convexity or concavity of the principal's (bilevel) optimization problem,  assuming each agent's disutility is a strictly increasing but not
necessarily affine (i.e.\ quasilinear) function of the price paid. Concavity when present,  makes the problem more amenable to 
computational and theoretical analysis;  it is key to obtaining uniqueness and stability results for the principal's strategy in particular.  Even in the quasilinear case, our analysis goes beyond previous work by addressing convexity as well as concavity,  by establishing 
conditions which are not only sufficient but necessary,  and by requiring fewer hypotheses
on the agents' preferences.
\end{abstract}
\bigskip
\maketitle

\section{Introduction}

As one of the central problems in microeconomic theory, the {\em principal-agent framework} characterizes the type of non-competitive decision-making problems which involve aligning incentives so that one set of parties (the agents) finds it beneficial to act in the interests of another (the principal) despite holding private information.
It arises in a variety of different contexts. Besides nonlinear pricing \cite{Armstrong96,MussaRosen78,Spence80,Wilson93}, economists also use this framework to model many different types of transactions, including tax policy \cite{GuesnerieLaffont78,Mirrlees71,Rochet85}, contract theory \cite{QuinziiRochet85}, regulation of monopolies \cite{BaronMyerson82}, product line design \cite{RochetChone98}, labour market signaling \cite{Spence74}, public utilities \cite{Roberts79}, and mechanism design \cite{MaskinRiley84, Mirrlees71, MonteiroPage98, Myerson81, Vohra11}. Many of these share the same mathematical model. In this paper, we use nonlinear pricing to motivate the discussion,  in spite of the fact that our conclusions
may be equally pertinent to many other areas of application.\\

Consider the problem of a monopolist who wants to maximize her profits by selecting the dependence of the price $v(y)$ on each type  $y \in cl(Y)$ of product sold. An agent of type $x \in X$ will choose to buy that product which maximizes his benefit 

\begin{equation}\label{1}
u (x) := \max_{y \in cl(Y)} G(x, y, v(y))
\end{equation}
where $(x, y, z)\in X \times cl(Y)\times \R \longmapsto G(x,y,z) \in \R$, is the given direct utility function for agent type $x$ to choose product type $y$ at price $z$, and $X,Y$ are open and bounded subsets in $\R^m$ and $\R^n$ ($m \ge n$), respectively,
with closures $cl(X)$ and $cl(Y)$.
 After agents, whose distribution $d \mu(x)$ is known to the monopolist, have chosen their favorite items to buy, the monopolist calculates her profit to be given by the functional 
\begin{equation}\label{profit}
	\Pi(v, y):=\int_{X} \pi(x,y(x),v(y(x))) d\mu(x),
\end{equation}
 where $y(x)$ denotes the product type $y$ which agent type $x$ chooses to buy
(and which maximizes \eqref{1}), $v(y(x))$ denotes the selling price of type $y(x)$ and $\pi \in C^0(cl(X\times Y)\times \R)$ denotes the principal's net profit of selling product type $y \in cl(Y)$ to agent type $x \in X$ at price $z \in \R$. The monopolist wants to maximize her net profit among all  lower semicontinuous pricing policies. Individual agents accept to contract only if the benefits they earn are no less than their outside option; we model this by assuming
the existence of a distinguished point $y_\nul \in cl(Y)$ which represents the outside option,
and whose price  cannot exceed some fixed value $z_\nul \in \R$ beyond the monopolist's control.
This removes any incentive for the monopolist to raise the prices of other options too high.
(We can choose normalizations such as $\pi(x,y_\nul,z_\nul)=0=G(x,y_\nul,z_\nul)$ and $(y_\nul,z_\nul)=(0,0)$, or not, as we wish.)
\\

The following is a table of notation:\\
\\
\begin{minipage}{\linewidth}
	\captionof{table}{Notation} \label{tab:title} 
	\begin{tabular}{ C{1.12in} C{.1in} *4{C{1.15in}} }\toprule[1.5pt]
	 Mathematical Expression & \text{Economic Meaning} & & & & \\ 
	 \midrule[0.5pt]
	 $x$  & \text{agent type}\\
	 $y$  & \text{product type}\\
	 $X \subset \R^m$  & \text{(open, bounded) domain of agent types} \\
	 $cl(Y) \subset \R^n$  & \text{domain of product types, closure of $Y$} \\
	 $v(y)$  & \text{selling price of product type $y$}\\
	$v(y_\nul) \le z_\nul$ & \text{price normalization of the outside option $y_\nul \in cl(Y)$}\\ 
	 $u(x)$  & \text{indirect utility of agent type $x$}\\
         $\dom Du$ & \text{points in $X$ where $u$ is differentiable} \\ 
	 $G(x,y,z)$ & \text{direct utility of buying product $y$ at price $z$ for agent $x$ }\\
         $H(x,y,u)$ & \text{price at which $y$ brings $x$ utility $u$, so that $H(x,y,G(x,y,z))=z$} \\
	 $\pi(x,y,z)$  & \text{the principal's profit for selling product $y$ to agent $x$ at price $z$ }\\
	 $d\mu(x)$ & \text{Borel probability measure giving the distribution of agent types on $X$}\\
          $\mu  \ll {\mathcal L}^m$ & \text{$\mu$ vanishes on each subset of $\R^m$ having zero Lebesgue volume ${\mathcal L}^m$} \\
	 $\Pi(v,y)$ & \text{monopolist's profit facing agents' responses $y(\cdot)$ to her chosen price policy $v(\cdot)$} \\	${\pmb\Pi}(u)$ & \text{monopolist's profit, viewed instead as a function of agents' indirect utilities $u(\cdot)$ } \\

		\bottomrule[1.25pt]
		\end {tabular}\par
		\bigskip
	\end{minipage}
	
	\bigskip
%
%

 For the quasilinear case, where the utility $G(x,y,z)$ depends linearly on its third variable, and net profit $\pi(x,y,z)=z-a(y)$ represents difference of selling price $z$ and manufacturing cost $a$ of product type $y$, theories of existence \cite{Basov05,RochetStole03,Carlier01,MonteiroPage98}, uniqueness 
\cite{CarlierLachand-Robert01,FigalliKimMcCann11,MussaRosen78,RochetChone98} 
and robustness \cite{Basov05,FigalliKimMcCann11} have been well studied. The equivalence of concavity to the corresponding non-negative cross-curvature condition revealed by Figalli-Kim-McCann \cite{FigalliKimMcCann11} directly inspires our work. In addition to the quaslinearity of
$G(x,y,z) = b(x,y) - z$ essential to their model,  they require additional restrictions such as $m=n$ and $b \in C^4(cl(X\times Y))$ which are not economically motivated
and which we shall relax or remove.  Moreover,  our results allow for the monopolist's profit $\pi$ to depend in a general
way both on monetary transfers and on the agents' types $x$,  revealed after contracting.  Such dependence plays an important role in applications such as insurance marketing.
\\
 
 For a particular case, where $ cl(X) = cl(Y) = [0,\infty)^n$, $G(x,y,z) = \langle x,y \rangle -z$,  $a(y) =  
\langle y,y\rangle /2$, and $(y_\nul,z_\nul)=(0,0)$, Rochet and Chon$\acute{e}$  \cite{RochetChone98} not only obtained existence results but also partially characterized optimal solutions and interpreted them economically. Here $\langle\ ,\ \rangle$ denotes the Euclidean inner product.\\
 
 More generally, Carlier \cite{Carlier01} proved existence for the general quasilinear utility, where $X$ is a bounded open convex subset in $\R^m$ with $C^1$ boundary, $cl(Y)= [0,\infty)^n$, $G(x,y,z) = b(x,y)- z$ 
 and the manufacturing cost $a$ is linear; see also Monteiro and Page \cite{MonteiroPage98}. \\
 
The generalization of quasilinear to nonlinear preferences has many potential applications
. For example, the benefit function $G(x,y,v(y))=b(x,y)-v^2(y)$ models agents who are more sensitive to high prices, while $G(x,y,v(y))=b(x,y)-v^{\frac{1}{2}}(y)$ models agents who are less sensitive to high prices. 
Very few results are known for such nonlinearities,  due to the complications which they entail.
However,  we shall eventually show that under certain conditions the concavity or convexity of $G$ and $\pi$ (or their derivatives)
with respect to $v$ tends to be reflected by concavity or convexity of $\Pi$, not with respect to 
$v$ or $y$,  but rather with respect to the agents indirect utility $u$, in terms of 
which the principal's maximization is reformulated below. \\

The generalized existence problem was mentioned as a conjecture by Basov \cite[Chapter 8]{Basov05}. 
Recently, N\"oldeke and Samuelson \cite{NoldekeSamuelson15p} provided a general existence result for $cl(X)$, $cl(Y)$ being compact and the utility $G$ being decreasing with respect to its third variable, by implementing a duality argument based on Galois Connections.
 Initially independently of \cite{NoldekeSamuelson15p}, Zhang \cite{Zhang} found another general existence result for the principal-agent framework where the utility function is decreasing with respect to its third variable but on potentially unbounded domains.  In Zhang's work, $G$-convexity plays a crucial role. 
\\
 
In 2013, Trudinger's lecture at the optimal transport program at MSRI inspired us to try generalizing Figalli-Kim-McCann \cite{FigalliKimMcCann11} to the non-quasilinear case. With the tools developed by Trudinger \cite{Trudinger14} and others  \cite{Balder77,Singer97}, we are able to provide convexity and concavity theorems for general utility and net profit functions. For an application of $G$-convexity to geometric optics, see \cite{GuillenKitagawaCPAM}.\\

Section 2 lists most of hypotheses we need in this paper. Section 3 introduces $G$-convexity, $G$-subdifferentiability, and various equivalent forms of the principal's program. Section 4 presents a variety of necessary and sufficient conditions for concavity (and convexity) of the principal's problem, and the resulting uniqueness of her optimal
strategy.  Several examples are developed, based on an analytic criterion for concavity, whose proof is deferred to section 5. 
Appendix A gives a differential criterion for the crucial hypotheses --- \Gthree\ of the next section --- clarifying its relation to that of Figalli, Kim \& McCann \cite{FigalliKimMcCann11}, and the Ma-Trudinger-Wang criteria for regularity of optimal maps \cite{MaTrudingerWang05} which inspired it.
 In appendix B,  we assume the monopolist's utility does not depend on the agent's private information, 
which in certain circumstances allows us to provide a necessary and sufficient condition for concavity of her profit functional.

\bigskip

\section{Hypotheses}

For notational convenience, we adopt the following technical hypotheses, inspired by those of 
Trudinger~\cite{Trudinger14} and Figalli-Kim-McCann \cite{FigalliKimMcCann11}. Here we use $G_x=\big(\frac{\partial G}{\partial x^1}, \frac{\partial G}{\partial x^2}, ..., \frac{\partial G}{\partial x^m}\big)$, $G_y=\big(\frac{\partial G}{\partial y^1}, \frac{\partial G}{\partial y^2}, ..., \frac{\partial G}{\partial y^n}\big)$, $G_z = \frac{\partial G}{\partial z}$ to denote derivatives with respect to $x\in X\subset \R^m$ , $y \in Y \subset \R^n$, and $z\in \R$, respectively. Also, for second partial derivatives, we adopt following notation
\begin{flalign*}
	G_{x,y}=
	\begin{bmatrix}
	\frac{\partial^2 G}{\partial x^1 \partial y^1} & \frac{\partial^2 G}{\partial x^1 \partial y^2} & ... & \frac{\partial^2 G}{\partial x^1 \partial y^n}  \\
	\frac{\partial^2 G}{\partial x^2 \partial y^1} & \frac{\partial^2 G}{\partial x^2 \partial y^2} & ... & \frac{\partial^2 G}{\partial x^2 \partial y^n}  \\	
	\vdots & \vdots & \ddots & \vdots \\
	\frac{\partial^2 G}{\partial x^m \partial y^1} & \frac{\partial^2 G}{\partial x^m \partial y^2} & ... & \frac{\partial^2 G}{\partial x^m \partial y^n}  
	\end{bmatrix},
\end{flalign*}
$G_{x,z}=\big(\frac{\partial^2 G}{\partial x^1 \partial z}, \frac{\partial^2 G}{\partial x^2 \partial z}, ..., \frac{\partial^2 G}{\partial x^m \partial z}\big)$.\\ 

We say $G \in C^{1}(cl(X\times Y \times Z))$, if all the partial derivatives $\frac{\partial G}{\partial x^1}$, ...,$\frac{\partial G}{\partial x^m}$, $\frac{\partial G}{\partial y^1}$, ..., $\frac{\partial G}{\partial y^n}$, $\frac{\partial G}{\partial z}$ exist and are continuous. Also, we say $G \in C^{2}(cl(X\times Y \times Z))$, if all the partial derivatives up to second order( i.e. $\frac{\partial^2 G}{\partial \alpha \partial \beta}$, where $\alpha, \beta = x^1, ... , x^m, y^1, ..., y^n, z$) exist and are continuous. Any bijective continuous function whose inverse is also continuous, is called a homeomorphism (a.k.a.\ bicontinuous).\\

The following hypotheses will be relevant:  \Gone--\Gthree\ represent partial analogs of the twist, domain convexity,
and non-negative cross-curvature hypotheses from the quasilinear setting \cite{FigalliKimMcCann11} \cite{Loeper09};
\Gfour\ encodes a form of the desirability of money to each agent, while \Gfive\ quantifies the assertion that the maximum
price $\bar z$ is high enough that no agent prefers paying it for any product $y$ to the outside option.

\begin{itemize}
	\item[\Gzero] $G \in C^{1}(cl(X\times Y \times Z))$, where $X\subset \R^m, Y \subset \R^n$ are open and bounded and $Z=(\underbar z,\bar z)$ with $-\infty <\underbar z < \bar z \le +\infty$.\\

	\item[\Gone] For each $x \in X$, the map $(y,z) \in cl( Y \times Z) \longmapsto (G_x, G)(x,y,z)$ is a homeomorphism onto its range;\\
	\item[\Gtwo] its range $(cl( Y \times Z))_x := (G_x,G)(x,cl(Y \times Z)) \subset \R^{m+1}$ is convex.\\
	\item[] For each 
	$x_0\in X$ and $(y_0, z_0),(y_1,z_1) \in cl( Y \times Z)$, 
	define $(y_t, z_t)\in cl( Y \times Z)$ such that the following equation holds:
	\begin{equation}\label{$G$-segment}
	(G_x, G)(x_0,y_t,z_t) = (1-t)(G_x, G)(x_0,y_0,z_0)+t (G_x, G)(x_0,y_1,z_1), \text{ for each $t\in [0,1].$}
	\end{equation}
	By \Gone \ and \Gtwo , $(y_t, z_t)$ is uniquely determined by (\ref{$G$-segment}). 
	We call $t \in [0,1] \longmapsto (x_0,y_t,z_t)$ the  $G$-segment connecting $(x_0, y_0, z_0)$ and $(x_0, y_1, z_1)$.\\
	\item[\Gthree] For each $x,x_0 \in X$, assume $t \in [0,1] \longmapsto G(x, y_t, z_t)$ is convex along all $G$-segments ($\ref{$G$-segment}$).\\
	\item[\Gfour]  For each $(x,y,z) \in X \times cl(Y)\times cl(Z)$, assume $G_{z}(x,y,z)<0$.  \\ 
\item[\Gfive] $\pi\in C^0(cl(X\times Y \times Z))$ and $u_\nul (x) := G(x,y_\nul,z_\nul)$ for some fixed
$(y_\nul,z_\nul) \in cl(Y \times Z)$  satisfying
$$G(x,y,\bar z)  := \lim_{z \to \bar z}G(x,y,z) \le G(x,y_\nul,z_\nul) \text{ for all}\ (x,y) \in X \times cl (Y).$$
When $\bar z = +\infty$ assume this inequality is strict, and moreover that $z$ sufficiently large implies
$$G(x,y,z) < G(x,y_\nul,z_\nul) \text{ for all}\ (x,y) \in X \times cl(Y).$$

\end{itemize}

For each $u \in \R$, \Gfour \ allows us to define $H(x,y,u) := z$ 
if  $G(x,y,z) = u$, i.e. $H(x, y, \cdot)= G^{-1}(x,y,\cdot)$.

\bigskip

\section{Principal's program, $G$-convexity, $G$-subdifferentiability}\label{setting}


In this section, we will introduce the principal's program and reformulate it in the language of $G$-convexity and $G$-subdifferentiability. The purpose of this section is to fix terminology and prepare the preliminaries 
for the main results of the next section.

In economic models, 
incentive compatibility is  needed to ensure that all the agents report their preferences truthfully. 
{ According to the revelation principle, this costs no generality.
Decisions made by monopolist according to the information collected from agents then lead} 
to the expected market reaction (as in \cite{Carlier01,RochetChone98}).  
 Individual rationality
is required to ensure full participation, so that each agent will choose to play,  possibly by accepting the outside option.

\begin{definition}[Incentive compatible and individually rational]
A measurable map $x \in X \longmapsto (y(x),z(x)) \in cl(Y \times Z)
$ of agents to (product, price) pairs
is called {\em incentive compatible} if and only if $G(x,y(x),z(x)) \ge G(x, y(x'), z(x'))$ for all $(x,x')\in X$.
Such a map offers agent $x$ no incentive to pretend to be $x'$.
 It is called {\em individually rational} if and only if $G(x,y(x),z(x)) \ge G(x,y_\nul,z_\nul)$ for all $x \in X$,
meaning no individual $x$ strictly prefers the outside option to his assignment $(y(x),z(x))$.
\end{definition}

Given $\pi$ and $(y_\nul,z_\nul)$ from \Gfive, the principal's program can be described as follows:

\begin{equation}
(P_0)
\begin{cases}
\sup \Pi(v,y)=\int_{X} \pi(x, y(x), v(y(x))) d\mu(x)\quad \text{among}\\ 
x \in X \longmapsto (y(x),v(y(x))) \text{  incentive compatible,  individually rational,}\\ 
{ \text{and}\ v:cl(Y)\longrightarrow cl (Z)\  \text{lower semicontinuous with}\ v(y_\nul) \le z_\nul}. \\
\end{cases}
\end{equation}

In this section we'll also see that incentive compatibility is conveniently encoded via the {\em $G$-convexity}
defined below of the agents' indirect utility $u$.  We'll then reformulate the 
principal's program using $u$ as a proxy for the prices $v$ controlled
by the principal,  thus generalizing Carlier's approach \cite{Carlier01} to the non-quasilinear setting. Moreover, the agent's indirect utility $u$ and product selling price $v$ are $G$-dual to each other in the sense of \cite{Trudinger14}.

\begin{definition}[$G$-convexity]\label{$G$-convexity}
A function $u\in C^0(X)$ is called $G$-convex if for each $x_0 \in X$, there exists $y_0 \in cl(Y)$, and $z_0 \in cl(Z)$ such that $u(x_0)=G(x_0, y_0, z_0)$, and $u(x)\ge G(x, y_0, z_0),\mbox{ for all } x\in X$.
\end{definition}


From the definition, we know if $u$ is a $G$-convex function, for any $x \in X$ where $u$ happens
to be differentiable,  denoted $x\in \dom Du$, there exists $y\in cl(Y)$ and $z\in cl(Z)$ such that
\begin{equation}\label{EqnInverse}
u(x)= G(x, y, z),\ \ \   Du(x) = D_x G(x, y, z).
\end{equation}
Conversely, when (\ref{EqnInverse}) holds, one can identify $(y, z) \in cl( Y \times Z)$ in terms of $u(x)$ and $Du(x)$, according to Condition \Gone. We denote it as $\bar{y}_G (x,u(x),Du(x)) := (y_G, z_G)(x,u(x),Du(x))$, 
and drop the subscript
$G$ when it is clear from context. 
{Under our hypotheses, $\bar y_G$ is a continuous function, at least in the region where $u \ge u_\emptyset$.}
It will often prove convenient to augment the types $x$ and $y$ with an extra real variable;
here and later we use the notation $\bar x \in \R^{m+1}$ and $\bar y \in \R^{n+1}$ to signify this augmentation.
In addition, the set $X\setminus \dom Du$ has Lebesgue measure zero, which will be shown in the proof of Theorem \ref{Thm:Existence}.

While $G$-convexity acts as a generalized notion of convexity, the $G$-subdifferentiability defined below generalizes the concept of differentiability/subdifferentiability.

\begin{definition}[$G$-subdifferentiability]
	The $G$-subdifferential of a 
	function $u(x)$ is defined by
	\begin{equation*}
	\partial^G u(x):= \{ y\in cl(Y)\mid u(x')\ge G(x',y, H(x,y,u(x))), \mbox{ for all } x'\in X\}.
	\end{equation*}
	
	A function $u$ is said to be $G$-subdifferentiable at $x$ if and only if $\partial^G u(x) \neq \emptyset$.
\end{definition}

In \cite{Trudinger14}, this point-to-set map $\partial^G u$ is also called $G$-$normal$ mapping. For more properties 
related to $G$-convexity, see \cite{Trudinger14}.

Lemma \ref{convex-subdiff} shows an equivalent relationship between $G$-convexity and $G$-subdifferentiability, a special case of which is that a function is convex if and only if it is subdifferentiable on (the interior of) its domain. 

\begin{lemma}[$G$-subdifferentiability characterizes $G$-convexity]
	\label{convex-subdiff}
	A function $u: X \longrightarrow \R$ is $G$-convex if and only if it is $G$-subdifferentiable everywhere.
\end{lemma}
\begin{proof}[Proof
	]
	Assume $u$ is $G$-convex, want to show $u$ is $G$-subdifferentiable everywhere, i.e., need to prove  $\partial^G u(x_0)\neq \emptyset$ for all $x_0\in X$.
	
	Since $u$ is $G$-convex, by definition, for each $x_0$, there exists $ y_0, z_0$, such that $u(x_0) = G(x_0,y_0,z_0)$, and for all  $ x \in X, u(x)\ge G(x, y_0, z_0) = G(x, y_0, H(x_0,y_0,u(x_0)))$. By the definition of $G$-subdifferentiability, $y_0 \in \partial^G u(x_0)$, i.e. $\partial^G u(x_0) \neq \emptyset$.

	On the other hand, assume $y$ is $G$-subdifferentiable everywhere, then for each $ x_0 \in X$,  there exists $ y_0 \in \partial^G u(x_0)$. Set $z_0:=H(x_0,y_0,u(x_0))$ so that $u(x_0) 
	= G(x_0, y_0, z_0)$.
	
	Since $y_0\in \partial^G u(x_0)$, we have $u(x)\ge G(x,y_0,H(x_0,y_0,u(x_0))) = G(x,y_0,z_0)$ 
	for all $x\in X$.
	By definition, $u$ is $G$-convex.
\end{proof}

We now show each $G$-convex function defined in Definition \ref{$G$-convexity} can be achieved by 
some price menu $v$,  and conversely each price menu yields a $G$-convex indirect utility \cite{Trudinger14}.
We require either \Gfive\ or \eqref{repulsed}, which asserts
all agents are repelled by the maximum price,  and insensitive to which contract they receive at that price.


	\begin{proposition}[Duality between prices and indirect utilities]\label{Prop:Gtransform}
	Assume \Gzero\ and \Gfour. {\rm (a)} If 
\begin{equation}\label{repulsed}
G(x,y,\bar{z}) := \lim_{z \to \bar z}G(x,y,z) = \inf_{(\tilde{y}, \tilde{z})\in cl(Y\times Z)} G(x, \tilde{y}, \tilde{z}) 
\qquad {\forall} (x,y)\in X\times cl(Y),
\end{equation} 
then a function $u \in C^0(X)$ is $G$-convex if and only if there exist a lower semicontinuous 
$v: cl(Y) \longrightarrow cl(Z)$ such that $u(x) = \max_{y\in cl(Y)} G(x,y,v(y))$.
{\rm (b)} If instead of \eqref{repulsed} we assume \Gfive,
then $u_\nul \le u \in C^0(X)$ is $G$-convex if and only if there exists a lower semicontinuous function $v: cl(Y) \longrightarrow cl(Z)$ with $v(y_\nul) \le z_\nul$ such that $u(x) = \max_{y\in cl(Y)} G(x,y,v(y))$.
		\end{proposition}

\begin{proof}
	1. Suppose $u$ is $G$-convex. Then for any agent type $x_0\in X$, there exists a product and price
$(y_0,z_0) \in cl(Y \times Z)$, such that $u(x_0)=G(x_0,y_0,z_0)$ and $u(x)\ge G(x,y_0,z_0)$, for all $x\in X$. 
			
{
	Let $A:=\cup_{x\in X} \partial^{G} u(x)$ denote the corresponding set of products.          
	For $y_0\in A$, define $v(y_0) = z_0$, where $z_0\in cl(Z)$ and $x_0\in X$ satisfy $u(x_0)=G(x_0,y_0,z_0)$ and $ u(x)\ge G(x, y_0,z_0)$ for all $x\in X$.
	We shall shortly show this makes $v:A \longrightarrow cl(Z)$ (i) well-defined and (ii) lower semicontinuous. Taking (i)  for granted,  our construction yields
	\begin{equation}\label{restricted1}
	u(x)  =
	\max_{y \in A} G(x,y,v(y)) \qquad \forall x \in X.
	\end{equation}
	
	(i) Now for $y_0 \in A$, suppose there exist $(x_0, z_0), (x_1, z_1)\in X\times cl(Z)$ with 
	$z_0 \ne z_1$, such that $u(x_i) = G(x_i, y_0, z_i)$ and $u(x) \ge G(x, y_0, z_i)$ for all $x\in X$ and $i=0,1$.  Without loss of generality, assume $z_0< z_1$. By \Gfour, we know $u(x_1) = G(x_1, y_0, z_1)<G(x_1, y_0, z_0)$, contradicting $u(x)\ge G(x,y_0,z_0)$, for all $x\in X$.   
	Having shown $v:A \longrightarrow cl(Z)$ is well-defined, we now show it is lower semicontinuous.
}

%

{ (ii)} Suppose $\{y_{k}\}\subset A$ converges to $y_0 \in A$ and 
$z_\infty:= \lim\limits_{k \to \infty} v(y_k) = \liminf\limits_{y \rightarrow y_0}v(y)$.  We need to show 
$v(y_0) \le z_\infty$.  Letting $z_k := v(y_k)$ for each $k$,  there exists $x_k \in X$ such that 
\begin{equation}\label{sequence3.6}
u(x) \ge G( x, y_k,z_k) \qquad {\forall} x \in X \text{ and}\ k =0,1,2,\ldots,
\end{equation}
with equality holding at $x=x_k$.  In case (b) we deduce $z_\infty < \infty$ from 
$$
G(x_k,y_k,z_k) =u(x_k)  \ge G(x_k,y_\nul,z_\nul)
$$
and \Gfive.
Taking $k \to \infty$,  \Gzero\ (or \eqref{repulsed}  in case (a) when $z_\infty=+\infty$)
implies
\begin{equation}\label{limit3.6}
u(x) \ge G(x,y_0,z_\infty) \qquad {\forall} x \in X.
\end{equation}
Applying \Gfour\ to $G(x_0,y_0,z_0)=u(x_0)\ge G(x_0,y_0,z_\infty)$ yields 
the desired semicontinuity: $z_0 \le z_\infty$.

{(iii)} We extend $v$ from $A$ to $cl(Y)$ by taking its lower semicontinuous hull;
this does not change the values of $v$ on $A$, but satisfies $v(y_0) := \bar z$  on $y_0 \notin cl(A)$.
We now 
show this choice of price menu $v$ yields \eqref{1}. 
Recall for each $x \in X$,  there exists $(y_0,z_0) \in cl(Y \times Z)$ such that
$$
u(x) = G(x,y_0,z_0) \ge (u_\nul(x) := G(x,y_\nul,z_\nul) \ge) \sup_{y \in cl(Y)\setminus cl(A)} G(x,y,v(y)),
$$
in view of \eqref{repulsed} (or \Gfive), and the fact that $v(y) = \bar z$ for each $y$ outside $cl(A)$.
Thus to establish \eqref{1},  we need only show that \eqref{restricted1} remains true when
the domain of the maximum is enlarged from $A$ to $cl(A)$.
Since we have chosen the {\em largest} lower semicontinuous extension of $v$ outside of $A$,
each $y_0 \in cl(A) \setminus A$ is approximated by a sequence  $\{y_{k}\}\subset A$
for which $z_k := v(y_k)$ converges to $z_\infty :=v(y_0)$. As before, \eqref{sequence3.6} holds
 and implies \eqref{limit3.6}, showing \eqref{restricted1} indeed remains true when
the domain of the maximum is enlarged from $A$ to $cl(A)$, and establishing \eqref{1}.
Finally, if $v(y_\nul) > z_\nul$ in case (b) then \Gfour\ yields $u(x) \ge u_\nul(x) > G(x,y_\nul,v(y_\nul))$, and
we may redefine $v(y_\nul):=z_\nul$ without violating either \eqref{1} or the lower semicontinuity of $v$.

	2. Conversely, suppose there exist a lower semicontinuous function $v: cl(Y)\longrightarrow cl(Z)$, such that $u(x)=\max_{y\in {cl(Y)}} G(x, y,v(y))$. Then for any $x_0\in X$, there exists $y_0 \in cl(Y)$, such that $u(x_0) = G(x_0, y_0, v(y_0))$. Let $z_0:= v(y_0)$, then $u(x_0)= G(x_0, y_0, z_0)$, and for all $x\in X$, $u(x)\ge G(x, y_0, z_0)$. By definition, $u$ is $G$-convex.   If $v(y_\nul) \le z_\nul$ then $u(\cdot) \ge G(\cdot ,y_\nul,v(y_\nul)) \ge u_\nul(\cdot)$ 
by \eqref{1} and \Gfour.
\end{proof}

\begin{remark}[Optimal agent strategies] Assume \Gzero\ and \Gfour.
When $\bar z< \infty$,  lower semicontinuity of $v:cl(Y) \longrightarrow cl(Z)$ is enough to ensure the maximum 
\eqref{1} is attained.  However, when $\bar z=+\infty$ we can reach the same conclusion either by assuming the limit \eqref{repulsed} converges uniformly with respect to  $y \in cl(Y)$,  or else by assuming $v(y_\nul) \le z_\nul$ and \Gfive.
\end{remark}

\begin{proof}
	 For any fixed $x \in X$ let $u(x) = \sup\limits_{y\in cl(Y)} G(x,y,v(y))$. We will show that the maximum is attained. Since $cl(Y)$ is compact, suppose $\{y_{k}\}\subset cl(Y)$ converges to $y_0 \in cl(Y)$,
	$z_\infty:= \limsup\limits_{k \to \infty} v(y_k)$ and $u(x) = \lim\limits_{k\to \infty} G(x, y_k,  v(y_k))$. By extracting subsequence of $\{y_k\}$ and relabelling, without loss of generality, assume $\lim\limits_{k\to \infty} v(y_k) = z_{\infty}$. \\
	
	1.  If $z_{\infty}< \bar{z}$ then lower semicontinuity of $v$ yields $v(y_0)\le z_{\infty} < +\infty$. By \Gfour, one has
	\begin{equation}
		G(x, y_0, v(y_0)) \ge G(x, y_0, z_{\infty}) = \lim\limits_{k\to \infty} G(x, y_k,  v(y_k)) = u(x) = \sup\limits_{y\in cl(Y)} G(x,y,v(y)).
	\end{equation}
	Therefore, the maximum is attained by $y_0$.
	
	2. If $z_{\infty} = \bar{z}$ then $\lim\limits_{k\to \infty} v(y_k) = \bar{z} = + \infty$.
	
	2.1. By assuming the limit \eqref{repulsed} converges uniformly with respect to  $y \in cl(Y)$, we have 
	\begin{equation*}
		\inf_{(\tilde{y}, \tilde{z})\in cl(Y\times Z)} G(x, \tilde{y}, \tilde{z}) = G(x,y_0,\bar{z}) = \lim_{k \to \infty} G(x,y_k, v(y_k)) = u(x) = \sup \limits_{y\in cl(Y)} G(x,y,v(y)).
	\end{equation*}
	In this case, the maximum is attained by $y_0$. 
	
	2.2. By assuming \Gfive, for sufficient large $k$, we have $G(x, y_k, v(y_k)) < G(x, y_{\emptyset}, z_{\emptyset})$. Taking $k \to \infty$, by $v(y_{\emptyset}) \le z_{\emptyset}$  and \Gfour, one has
	\begin{equation*}
		\sup\limits_{y\in cl(Y)} G(x,y,v(y)) = u(x) = \lim_{k \to \infty} G(x,y_k, v(y_k)) \le G(x, y_{\emptyset}, z_{\emptyset}) \le G(x, y_{\emptyset}, v(y_{\emptyset})).
	\end{equation*}
	Thus, the maximum is attained by $y_{\emptyset}$.
	\end{proof}

Lemma \ref{L:incen/convex} plays the role of bridge connecting incentive compatibility in the economic context with $G$-convexity and $G$-subdifferentiability in mathematical analysis, generalizing the results of Rochet \cite{Rochet87} and Carlier \cite{Carlier01}.

\begin{lemma}[$G$-convex utilities characterize incentive compatibility]\label{L:incen/convex}
A measurable map $x\in X \longmapsto (y(x), z(x)) \in cl(Y \times Z) 
$ represents an incentive compatible contract if and only if $u(\cdot):=G(\cdot,y(\cdot),z(\cdot))$ is $G$-convex on $X$ and $y(x)\in \partial^G u(x)$ for each $ x \in X$.
\end{lemma}

\begin{proof}[Proof
]

Suppose $(y,z)$ is incentive compatible, then for each $ x_0 \in X$, let $y_0 = y(x_0)$, $z_0 = z(x_0)$, then $u(x_0) = G(x_0, y(x_0), z(x_0)) = G(x_0, y_0, z_0)$; moreover for all $ x\in X$, $u(x)=G(x, y(x), z(x)) \ge G(x, y(x_0), z(x_0)) = G(x,y_0,z_0)$. 
 By definition, $u$ is $G$-convex. 
	
	Since $u(x_0)=G(x_0, y_0, z_0)$, one has $z_0 = H(x_0, y_0, u(x_0))$.  Since for all $ x\in X$, $u(x) \ge G(x, y_0, z_0)$, thus $u(x)\ge G(x, y_0, H(x_0, y_0, u(x_0)))$. By definition, $y(x_0) = y_0 \in \partial^G u(x_0)$.
	
On the other hand, suppose $u(x)=G(x,y(x),z(x))$ is $G$-convex, and for each $ x\in X$, $y(x)\in \partial^G u(x)$,  so that $ u(x')\ge G(x', y(x), H(x, y(x), u(x)))=G(x', y(x), z(x))$ for all $ x'\in X$. Then $G(x',y(x'),z(x')) \ge G(x', y(x), z(x))$ for all $x, x'\in X$.
	By definition, (y,z) is incentive compatible.
\end{proof}


The following proposition not only reformulates the principal's problem, but manifests the existence of maximizer(s). 
For other existence results guaranteeing this supremum is attained in the non-quasilinear setting, 
see N\" oldeke-Samuelson \cite{NoldekeSamuelson15p} who require mere continuity of the direct utility $G$,
and Zhang~\cite{Zhang} who relaxes relative compactness of the domain.

In this paper, we use $\mathcal{L}^m$ to denote Lebesgue measure, which characterizes $m$-dimensional volume. A non-negative measure $\mu$ is said to be absolutely continuous with respect to  $\mathcal{L}^m$ if for every measurable set $A$, $\mathcal{L}^m(A) =0$ implies $\mu(A) =0$. This is written as $\mu \ll \mathcal{L}^m$. 

\begin{theorem}[Reformulating the principal's program using the agents' indirect utilities]\label{Thm:Existence}
Assume hypotheses \Gzero-\Gone and \Gfour--\Gfive, 
$\bar z < +\infty$  and $\mu 
\ll \mathcal{L}^m$. Setting \\ $\tilde{\Pi}(u,y)=\int_{X} \pi(x, y(x), H(x,y(x), u(x))) d\mu(x)$,
the principal's problem $(P_0)$ is equivalent to
\begin{equation}
(P)
\begin{cases}
\max \tilde{\Pi}(u,y) \\
\text{\rm among }
$G$\text{\rm-convex  $u(x) \ge u_\nul(x)$ with }
y(x) \in \partial^G u(x)\ \text{\rm for all } x \in X.\\
\end{cases}
\end{equation}
This maximum is attained. Moreover, $u$ determines $y(x)$ uniquely for a.e. $x \in X$.
\end{theorem}

\begin{proof}
{
1. Proposition \ref{Prop:Gtransform} encodes a bijective correspondence between
lower semicontinuous price menus $v:cl(Y) \longrightarrow cl(Z)$ with $v(y_\nul) \le z_\nul$
and $G$-convex indirect utilities $u \ge u_\nul$; it also shows \eqref{1} is attained.
Fix a $G$-convex $u \ge u_\nul$ and the corresponding price menu
$v$.  For each $x \in X$ let $y(x)$ denote the point achieving the maximum \eqref{1}, so that 
$u(x) = G(x,y(x),z(x))$ with $z(x):=v(y(x)) = H(x,y(x),u(x))$ and $\Pi(v,y) = \tilde \Pi(u, y)$.
From \eqref{1} we see
\begin{equation}\label{TIEGsub}
G(\cdot, y(\cdot),v\circ y(\cdot)) = 
u(\cdot) \ge G(\cdot, y(x),H(x,y(x),u(x))),
\end{equation}
so that $y(x) \in \partial^Gu(x)$.  Apart from the measurability established below,
Lemma \ref{L:incen/convex} asserts incentive compatibility of $(y,v \circ y)$,
while $u \ge u_\nul$ shows individual rationality, so $(P)\le(P_0)$.  
}

2. The reverse inequality begins with a lower semicontinuous price menu $v:cl(Y) \longrightarrow cl(Z)$ with 
$v(y_\nul) \le z_\nul$ and an incentive compatible, individually rational map $(y,v \circ y)$
on $X$. 
Lemma \ref{L:incen/convex}
then asserts $G$-convexity of $u(\cdot) :=G(\cdot, y(\cdot),v(y(\cdot)))$ and that $y(x) \in \partial^Gu(x)$ for each 
$x \in X$.  Choosing $\cdot = x$ in the corresponding inequality \eqref{TIEGsub}
produces equality, whence \Gfour\ implies
$v(y(x)) = H(x,y(x),u(x))$ and $\Pi(v,y) = \tilde \Pi(u, y)$.  Since $u \ge u_\nul$ follows from individual 
rationality, we have established equivalence of $(P)$ to $(P_0)$.  Let us now argue the supremum $(P)$ is attained.

{
%
%

	}
	3. Let us first show  $\pi(x, y(x), H(x,y(x), u(x)))$ is measurable on $X$ for all $G$-convex $u$ and $y(x) \in \partial^G u(x)$.\\
By \Gzero, we know $G$ is Lipschitz, i.e., there exists $L>0$, such that $|G(x_1, y_1, z_1)-G(x_2, y_2, z_2)|< L ||(x_1-x_2,y_1-y_2, z_1-z_2)||$, for all $(x_1,y_1,z_1), (x_2, y_2, z_2) \in cl(X\times Y\times Z)$.
Since $u$ is $G$-convex, for any $x_1, x_2 \in X$, there exist $(y_1, z_1), (y_2, z_2) \in cl( Y \times Z)$, such that $u(x_i) = G(x_i,y_i,z_i)$, for $i=1,2$. Therefore, 
\begin{equation*}
\begin{split}
&u(x_1)-u(x_2) \ge G(x_1, y_2, z_2) - G(x_2, y_2, z_2) > -L ||x_1-x_2||\\
& 	u(x_1)-u(x_2) \le G(x_1, y_1, z_1) - G(x_2, y_1, z_1) < L ||x_1-x_2||\\
\end{split}
\end{equation*}
That is to say, $u$ is also Lipschitz with Lipschitz constant $L$. By Rademacher's theorem and $\mu \ll \mathcal{L}^m$, we have $\mu(X\setminus \dom Du) =\mathcal{L}^m(X\setminus \dom Du) =0$.  Moreover, since $u$ is continuous, $\frac{\partial u(x)}{\partial x_j} = \lim\limits_{h\rightarrow 0} \frac{u(x+he_j)-u(x)}{h}$ is measurable on $\dom Du$, for $j=1,2,..., m$, where $e_j=(0,...0, 1, 0, ...,0)$ is the unit vector in $\R^m$ with $j$-th coordinate nonzero. Thus, $Du$ is also Borel on $\dom Du$.\\
	Since $y(x) \in \partial^G u(x) $, for all $x\in \dom Du$, we have
	\begin{equation}\label{EqnInverse2}
	u(x) = G(x, y(x), H(x, y(x), u(x))), \ \ \ Du(x) = D_xG(x, y(x), H(x, y(x), u(x))).
	\end{equation}
	By \Gone, there exists a continuous function $\yG
	$, such that $y(x) = \yG(x, u(x), Du(x))$. Thus $y(x)$ is Borel on $\dom Du$, which implies $\pi(x, y(x), H(x,y(x), u(x)))$ is measurable on $X$, given $\pi \in C^0(cl(X\times Y\times Z))$ and $\mu \ll \mathcal{L}^m$. Here we use the fact that $H$ is also continuous since $G$ is continuous and strictly decreasing with respect
to its third variable.\\

	4. 
To show the supremum is attained,
	let $\{u_k\}_{k\in \NN}$ be a sequence of $G$-convex functions , $u_k(x) \ge u_{\emptyset}(x)$ and  $y_k(x) \in \partial^G u_k(x)$ for any $x\in X$ and $k \in \NN$,  such that $\lim_{k\rightarrow \infty} \tilde{\Pi}(u_k, y_k) = \sup \tilde{\Pi}(u,y)$, among all feasible $(u,y)$. Below we construct a feasible pair $(u_{\infty}, y_{\infty})$ 
attaining the maximum.
\\
	
	 4.1. Claim: There exists $M>0$, such that $|u(x)|<M$, for any $G$-convex $u$ and any $x \in X$. Thus $\{u_k\}_{k \in \NN}$ is uniformly bounded.\\
	 Proof: Since $u$ is $G$-convex, for any $x\in X$, there exists $(y, z)\in cl( Y \times Z) $, such that $u(x) = G(x,y,z)$. Notice that $G$ is bounded, since $G$ is continuous on a compact set. Thus, there exists $M>0$, such that $|u(x)| =| G(x,y,z)|<M$ is also bounded.\\
	 
	 4.2. From part 1, we know $\{u_k\}_{k\in \NN}$ are uniformly Lipschitz with Lipschitz constant $L$, thus $\{u_k\}_{k\in \NN}$ are uniformly equicontinuous.\\
	
	4.3. By Arzel\`a-Ascoli theorem, there exists a subsequence of $\{u_k\}_{k\in \NN}$, again denoted as $\{u_k\}_{k\in \NN}$, and $u_{\infty} : X \longrightarrow \R$ such that $\{u_k\}_{k\in \NN}$  converges uniformly to $u_{\infty}$ on $X$.\\
	
	4.4. Claim: $u_{\infty}$ is also Lipschitz.\\
	Proof: For any $\varepsilon >0$, any $x_1, x_2 \in X$, since $\{u_k\}_{k\in \NN}$ converges to $u_{\infty}$ uniformly, there exist $K>0$, such that for any $k >K$, we have $|u_k(x_i) -u_{\infty}(x_i)|<\varepsilon$, for $i=1,2$. Therefore, 
	\begin{equation*}
		|u_{\infty}(x_1) -u_{\infty}(x_2)| \le |u_k(x_1)-u_{\infty}(x_1)| + |u_k(x_2)-u_{\infty}(x_2)| + |u_k(x_1) -u_k(x_2)| < 2\varepsilon + L ||x_1-x_2||.
	\end{equation*}
	Since the above inequality is true for all $\varepsilon >0$, thus $u_{\infty}$ is also Lipschitz.\\
	
	4.5. For any $x\in X$, since $u_k(x)\ge u_{\emptyset}(x)$ and $\lim\limits_{k\rightarrow \infty} u_{k}(x) = u_{\infty}(x)$, we have $u_{\infty}(x)\ge u_{\emptyset}(x)$. Therefore, $u_{\infty}$ satisfies the participation constraint.\\
	
	4.6. For any fixed $x\in X$, since $\{y_k(x)\}_{k \in \NN} \subset cl(Y)$ which is compact, there exists a subsequence $\{y_{k_l}(x)\}_{l\in \NN}$ which converges. Define $y_{\infty}(x):= \lim\limits_{l\rightarrow \infty} y_{k_l}(x) \in cl(Y)$.
	For each $l\in \NN$, because  $y_{k_l}(x)\in \partial^G u_{k_l} (x)$, by definition, we have $u_{k_l}(x_0)\ge G(x_0, y_{k_l}(x), H(x, y_{k_l}(x), u_{k_l}(x)))$, for any $x_0\in X$. This implies, for all $x_0\in X$, we have
	\begin{equation*}
	u_{\infty}(x_0) = \lim\limits_{l \rightarrow \infty} u_{k_l}(x_0) \ge \lim_{l\rightarrow \infty}  G(x_0, y_{k_l}(x), H(x, y_{k_l}(x), u_{k_l}(x))) \ge G(x_0, y_{\infty}(x), H(x, y_{\infty}(x), u_{\infty}(x))).
	\end{equation*} 
	Thus, $y_{\infty}(x) \in \partial^G u_{\infty}(x)$. \\
	Therefore, $\partial^G u_{\infty}(x) \ne \emptyset$, for any $x\in X$. By Lemma \ref{convex-subdiff}, this implies $u_{\infty}$ is $G$-convex.\\
	At this point, we have found a feasible pair $(u_{\infty}, y_{\infty})$, satisfying all the constraints in (P).\\
	
	4.7. Claim: For any $x\in \dom Du_{\infty}$, the sequence $\{y_k(x)\}_{k\in \NN} \subset cl(Y)$ converges to $y_{\infty}(x)$.\\
	Proof: Since $u_{\infty}$ is Lipschitz, by Rademacher's theorem, $u_{\infty}$ is differentiable almost everywhere in $X$, i.e. $\mu(X\setminus \dom Du_{\infty}) = \mathcal{L}^m(X\setminus \dom Du_{\infty}) =0$.\\
	For any $x\in \dom Du_{\infty}$ and any $\tilde{y}\in \partial^G u_{\infty}(x)$, we have $\tilde{y}(x) = \yG(x, u_{\infty}(x), Du_{\infty}(x))$, according to equation $(\ref{EqnInverse2})$ and hypothesis \Gone. This implies $ \partial^G u_{\infty}(x)$ is a singleton for each $x\in \dom Du_{\infty}$, i.e. $\partial^G u_{\infty}(x) =\{y_{\infty}(x)\}$.\\
	For any $x\in \dom Du_{\infty}$, by similar argument to that above in part 4.6, we can show that any (other) accumulation points of $\{y_k(x)\}_{k\in \NN}$ are elements in the set $\partial^G u_{\infty}(x)=\{y_{\infty}(x)\}$, i.e. the sequence $\{y_k(x)\}_{k\in \NN}$ converges to $y_{\infty}(x)$.\\
	
	4.8. Finally, since $\mu \ll \mathcal{L}^m$, by Fatou's lemma, we have 
	\begin{flalign*}
		\tilde{\Pi}(u_{\infty}, y_{\infty}) &= \int_{X} \pi(x, y_{\infty}(x), H(x,y_{\infty}(x), u_{\infty}(x))) d\mu(x)  \\
		&=\int_{X} \limsup\limits_{k\rightarrow \infty}\pi(x, y_k(x), H(x,y_k(x), u_k(x))) d\mu(x) \text{\hspace{1cm} (because $\mu \ll \mathcal{L}^m$)}\\
		&\ge \limsup\limits_{k\rightarrow \infty} \int_{X} \pi(x, y_k(x), H(x,y_k(x), u_k(x))) d\mu(x) \text{\hspace{1cm} (by Fatou's lemma)}\\
		& = \lim\limits_{k\rightarrow \infty} \tilde{\Pi}(u_k, y_k)\\
		&= \sup \tilde{\Pi}(u,y),
	\end{flalign*}
among all feasible (u,y). Thus, the supremum is attained.
\end{proof}

\begin{remark}[More Singular measures]
If $G \in C^2$ (uniformly in $z \in Z$) the same conclusions extend to $\mu$ which need not be absolutely
continuous with respect to Lebesgue,  provide $\mu$ vanishes on all hypersurfaces parameterized 
locally as a difference of convex functions \cite{FigalliKimMcCann11} \cite{Gigli11},  essentially because $G$-convexity then implies semiconvexity of $u$.  On the other hand,  apart from its final sentence,  
the proposition extends to all probability measures $\mu$ if $G$ is merely continuous, according to 
N\"oldeke-Samuelson \cite{NoldekeSamuelson15p}.  Our argument is simpler than
theirs on one point however:  
Borel measurability of $y(x)$ on $\dom Du$ follows automatically from $(G0)-(G1)$;
in the absence of these extra hypotheses,  they are required to make a measurable selection from among each
agent's preferred products to define $y(x)$.
\end{remark}

\begin{remark}[Tie-breaking rules for singular measures] When an agent $x$ finds more than one product which maximize his utility, in order to reduce the ambiguity, it is convenient to assume the principal has satisfactory persuasion to  convince the agent to choose one of those products which maximize the principal's profit. 
According to equation (\ref{EqnInverse}) and condition $(G1)$, this scenario would occur only for $x\in X \setminus \dom Du $, which has Lebesgue measure zero. Thus this convention has no effect for absolutely continuous measures, 
but can be used as in Figalli-Kim-McCann \cite{FigalliKimMcCann11} to extend our result to singular measures.
\end{remark}

\

\bigskip

\section{Concavity and Convexity Results}
The advantage of the reformulation from Section \ref{setting} is to make the principal's objective $\pmb \Pi$ 
depend on a scalar function $u$ instead of a vector field $y$.
By \Gone, the optimal choice $y(x)$ of   Lebesgue almost every agent $x\in X$
is uniquely determined by $u$.  Recall that $\bar{y}_G(x, u(x), Du(x))$ is the unique solution $(y,z)$ of the system ($\ref{EqnInverse}$), for any $x\in \dom Du$. 
Then the principal's problem (P) can be rewritten as maximizing a functional 
depending only on the agents' indirect utility $u$:
\begin{equation*}
	\max\limits_{ u\ge u_\nul \atop \text{$u$ is $G$-convex}} \pmb\Pi(u) := \max\limits_{u \ge u_\nul \atop \text{$u$ is $G$-convex}} \int_X \pi(x, \bar{y}_G(x, u(x), Du(x)))  d \mu(x) .
\end{equation*}
\\

Define $\mathcal{U}:=\{u:X\longrightarrow \R \mid  u \text{ is } G\text{-convex}\}$
and $\mathcal{U}_{\emptyset}:=\{u \in \mathcal{U}\mid  u\ge u_{\emptyset}\}$. 
Then the problem becomes to maximize $\pmb\Pi$ on $\mathcal{U}_\nul$. 
In this section, we give conditions under which the function space $\mathcal{U}_\nul$ 
is convex and the functional $\pmb\Pi$ is concave, often strictly. 
Uniqueness and stability of the principal's maximizing strategy follow from strict concavity as in \cite{FigalliKimMcCann11}. We also provide conditions under which $\pmb\Pi$ is convex. In this situation, the maximizers of $\pmb\Pi$ may not be unique, but are attained at extreme points of 
$\mathcal{U}_{\emptyset}$. (Recall that $u \in \mathcal{U}$ is called {\em extreme} if 
$u$ does not lie at the midpoint of any segment in $\mathcal{U}$.)
\\



	\begin{theorem}[$G$-convex functions form a convex set]\label{convex set}
		If $G: cl(X\times Y\times Z) \longrightarrow \R$ satisfies \Gzero-\Gtwo, then \Gthree\ becomes necessary and sufficient for the convexity of the set $\mathcal{U}$.
	\end{theorem}
\begin{proof}
	Assuming \Gzero-\Gtwo, for any $u_0, u_1 \in \mathcal{U}$, define $u_t(x) := (1-t)u_0(x)+tu_1(x)$, $t\in (0,1)$. We want to show $u_t$ is $G$-convex as well, for each $t\in (0,1)$. \\
	
	For any fixed $x_0 \in X$, since $u_0, u_1$ are $G$-convex, there exist $(y_0,z_0), (y_1, z_1) \in cl( Y \times Z)$, such that $u_0(x_0)= G(x_0, y_0,z_0)$, $u_1(x_0)=G(x_0,y_1,z_1)$, $u_0(x)\ge G(x,y_0,z_0)$ and $u_1(x)\ge G(x,y_1,z_1)$, for all $ x\in X$. \\ 
		
	Denote $(x_0, y_t, z_t)$ the $G$-segment connecting $(x_0, y_0, z_0)$ and $(x_0, y_1, z_1)$. Then $u_t(x_0) = (1-t)u_0(x_0)+tu_1(x_0)=(1-t) G(x_0, y_0,z_0)+ tG(x_0,y_1,z_1) = G(x_0,y_t, z_t)$, where the last equality comes from ($\ref{$G$-segment}$). \\
	
	In order to prove $u_t$ is $G$-convex, it remains to show $u_t(x)\ge  G(x,y_t,z_t)$, for all $x\in X$. \\
		
	By \Gthree, $G(x, y_t, z_t)$ is convex in $t$, i.e., $G(x,y_t,z_t)\le (1-t)G(x, y_0,z_0)+tG(x, y_1, z_1)$. So, $u_t(x)=$ $  (1-t)u_0(x)+tu_1(x) \ge (1-t) G(x,y_0,z_0)+tG(x, y_1, z_1)\ge G(x, y_t, z_t)$, for each $x \in X$. By definition, $u_t$ is $G$-convex, i.e., $u_t\in \mathcal{U}$, for all $t \in (0,1)$. Thus, $\mathcal{U}$ is convex.\\

 Conversely, assume $\mathcal{U}$ is convex. For any fixed $x_0 \in X$, $(y_t,z_t) \in cl( Y \times Z)$ with $(x_0, y_t, z_t)$ being a $G$-segment, we would like to show $G(x,y_t, z_t) \le (1-t)G(x,y_0,z_0)+tG(x,y_1, z_1)$, for any $x\in X$.\\
 

 Define $u_i(x) := G(x, y_i, z_i)$, for $i=0,1$. Then by definition of $G$-convexity, $u_0, u_1 \in \mathcal{U}$. Denote $u_t:= (1-t)u_0+tu_1$, for all $t\in (0,1)$. Since $\mathcal{U}$ is a convex set, $u_t$ is also $G$-convex. For this $x_0$ and each $t\in (0,1)$, there exists $(\tilde{y}_t, \tilde{z}_t) \in cl( Y \times Z)$, such that $u_t(x)\ge G(x, \tilde{y}_t, \tilde{z}_t)$, for all $x\in X$, and equality holds at $x_0$. Thus, $Du_t(x_0) = D_x G(x_0, \tilde{y}_t, \tilde{z}_t)$.\\
 
 Since $(x_0, y_t, z_t)$ is a $G$-segment, from $(\ref{$G$-segment})$, we know $D_x G(x_0, y_t, z_t) = (1-t) D_x G(x_0, y_0, z_0) +t D_x G(x_0, y_1, z_1) = (1-t) Du_0(x_0) +t Du_1(x_0) = Du_t(x_0)$. Thus, by \Gone,  $(\tilde{y}_t , \tilde{z_t}) = (y_t, z_t)$, for each $t\in (0,1)$. Therefore, $(1-t)G(x,y_0,z_0)+tG(x,y_1, z_1) = u_t \ge G(x, \tilde{y}_t, \tilde{z}_t) = G(x, y_t, z_t)$, for all $x\in X$, i.e., $G(x, y_t, z_t)$ is convex in $t$ along any $G$-segment $(x_0, y_t, z_t)$.
\end{proof}

	
 The following theorem provides a sufficient and necessary condition for the functional $\pmb \Pi(u)$ to be concave. It reveals the relationship between linear interpolations on the function space $\mathcal{U}$ and G-segments on the underlying type space $cl( Y \times Z)$.
	
\begin{theorem}[Concavity of the principal's objective]\label{maintheorem}
	If $G$ and $\pi: cl(X\times Y\times Z) \longrightarrow \R$ satisfy \Gzero-\Gfive, the following statements are equivalent:\\
$(i)$ $t\in[0,1] \longmapsto \pi(x, y_t ,z_t)$ is concave along all G-segments $(x, y_t, z_t)$;\\
$(ii)$ $\pmb \Pi(u)$ is concave in $\mathcal{U}$ for all $\mu\ll \mathcal{L}^m$. 
\end{theorem}
	
\begin{proof}
	$(i)\Rightarrow (ii).$ For any $u_0, u_1 \in \mathcal{U}$, $t\in (0,1)$, define $u_t = (1-t)u_0+tu_1$. We want to prove $\pmb \Pi(u_t) \ge (1-t) \pmb \Pi(u_0) + t\pmb \Pi(u_1)$, for any $\mu\ll \mathcal{L}^m$.\\
	Equations $(\ref{EqnInverse})$ implies that there exist $y_0, y_1: \dom Du \longrightarrow cl(Y)$ and $z_0, z_1: \dom Du \longrightarrow cl(Z)$ such that
	\begin{flalign}\label{Eqn:u_01}
	\begin{split}
		(G_x, G) (x, y_0(x), z_0(x)) &= (Du_0, u_0)(x),\\
		(G_x, G) (x, y_1(x), z_1(x)) &= (Du_1, u_1)(x).
	\end{split}
	\end{flalign}
	For each $x\in \dom Du$, $(y_0(x), z_0(x)), (y_1(x), z_1(x)) \in cl( Y \times Z)$, let $t\in [0,1] \longmapsto (x, y_t(x), z_t(x))$ be the $G$-segment connecting $(x, y_0(x), z_0(x))$ and $(x, y_1(x), z_1(x))$. Combining $(\ref{Eqn:u_01})$ and  $(\ref{$G$-segment})$, we have 
	\begin{equation}\label{EqnG-segments}
		(G_x, G) (x, y_t(x), z_t(x)) = (Du_t, u_t)(x).
	\end{equation}
	Thus, by concavity of $\pi$ on $G$-segments, for every $t \in [0,1]$,
	\begin{align*}
		\pmb \Pi (u_t)&= \int_X \pi(x, y_t(x), z_t(x))  d \mu (x)\\
		&\ge \int_X (1-t)\pi(x, y_0(x), z_0(x)) +t\pi(x, y_1(x), z_1(x))    d\mu (x)\\
		&=(1-t)\pmb \Pi (u_0)+ t \pmb \Pi(u_1).
	\end{align*}
	Thus, $\pmb \Pi $ is concave in $\mathcal{U}$.\\
	
$(ii)\Rightarrow (i).$ To derive a contradiction, assume $(i)$ fails.  Then there exists a $G$-segment $(x_0, y_t(x_0), z_t(x_0))$ and $t_0\in (0,1)$ such that $\pi(x_0, y_{t_0}(x_0), z_{t_0}(x_0)) <  (1-t_0)\pi(x_0, y_0(x_0),z_0(x_0))+ t_0 \pi(x_0, y_1(x_0), z_1(x_0))$. 
Let $u_0(x) := G(x, y_0(x_0), z_0(x_0))$, $u_1(x):= G(x, y_1(x_0), z_1(x_0))$ and $u_{t_0} = (1-t_0)u_0 +t_0 u _1$. Then $u_0, u_1, u_{t_0} \in \mathcal{U}$.  From $(\ref{Eqn:u_01})$ we know, $y_i(x)\equiv y_i(x_0)$, $z_i(x)\equiv z_i(x_0)$, for $i=0,1$.  Let $t\in [0,1] \longmapsto (x, y_t(x), z_t(x))$ be the $G$-segment connecting $(x, y_0(x), z_0(x))$ and $(x, y_1(x), z_1(x))$. And combining $(\ref{EqnInverse})$ and $(\ref{$G$-segment})$, we have
 \begin{flalign*}
 (G_x, G) (x, y_0(x_0), z_0(x_0)) &= (Du_0, u_0)(x),\\
 (G_x, G) (x, y_1(x_0), z_1(x_0)) &= (Du_1, u_1)(x).\\
 (G_x, G) (x, y_{t_0}(x), z_{t_0}(x)) &= (Du_{t_0}, u_{t_0})(x).
 \end{flalign*}

Since $\pi$, $y_{t_0}$ and $z_{t_0}$ are continuous, there exists $\varepsilon >0$, such that 
\begin{flalign*}
\pi(x, y_{t_0}(x), z_{t_0}(x)) < (1-t_0)\pi(x, y_0(x_0),z_0(x_0))  + t_0 \pi(x, y_1(x_0), z_1(x_0)), \text{ for all } x\in B_{\varepsilon}(x_0).
\end{flalign*} 
Here we use $B_{\varepsilon}(x_0)$ denote the open ball in $\R^m$ centered at $x_0$ with radius $\varepsilon$.
Take $d\mu = d\mathcal{L}^m
\mid _{B_{\varepsilon} (x_0)}/{\mathcal{L}^m} 
(B_{\varepsilon}(x_0))$ to be uniform measure on $B_\varepsilon(x_0)$. 
Thus, 
\begin{align*}
	\pmb \Pi (u_{t_0}) &= \int_X \pi(x, y_{t_0}(x), z_{t_0}(x)) d\mu(x)\\
	&<  \int_X    (1-t_0)\pi(x, y_0(x_0),z_0(x_0))  + t_0 \pi(x, y_1(x_0), z_1(x_0))     d\mu(x)\\
	&=(1-t_0)\pmb \Pi(u_0)+ t_0 \pmb \Pi(u_1).\\
\end{align*}
	This contradicts the concavity of $\pmb \Pi$.
\end{proof}

A similar proof shows the following result. Corollary $\ref{Cor:concave}$ implies that concavity of the principal's profit is equivalent to concavity of principal's utility along qualified $G$-segments. Moreover, Theorem~\ref{convex set} and Corollary $\ref{Cor:concave}$ together imply that the principal's profit $\pmb \Pi$ is a concave functional on a convex space, under assumptions \Gzero-\Gfive, $\mu\ll \mathcal{L}^m$, and $(i)'$ below.

\begin{corollary}\label{Cor:concave}
	If $G$ and $\pi$ satisfy \Gzero--\Gfive,  the following are equivalent:\\
	$(i)'$ $t\in[0,1] \longmapsto \pi(x, y_t(x) ,z_t(x))$ is concave  along all G-segments $(x, y_t(x), z_t(x))$ whose 
	endpoints satisfy $\min\{G(x, y_0(x), z_0(x)),G(x,y_1(x),z_1(x))\} \ge u_{\emptyset}(x)$;\\
	$(ii)'$ $\pmb \Pi(u)$ is concave in $\mathcal{U}_{\emptyset}$  for all $\mu\ll \mathcal{L}^m$. \\
\end{corollary}




To obtain uniqueness and stability of optimizers requires a stronger form of convexity.
Recall that a function $f$ defined on a convex subset of a normed space 
is said to be strictly convex if $f((1-t)x + ty)>(1-t) f(x) + tf(y)$ 
whenever $0<t<1$ and $x\ne y$. It is said to be (2-)uniformly concave, if there exists    
$\lambda>0$, such that for any $x, y$ in the domain of $f$ and $t\in [0,1]$, the following inequality holds.
\begin{flalign*}
	f((1-t)x+ty) -(1-t)f(x) - t f(y) \ge t(1-t)\lambda||x-y||^2.
\end{flalign*}
For such strengthenings,  it is necessary to view indirect utilities $u \in \mathcal U$ 
as equivalence classes
of functions which differ only on sets of $\mu$ measure zero.  More precisely, it is natural to adopt 
the Sobolev norm
$$
\|u\|^2_{W^{1,2}(X,d\mu)} := \int_X( |u|^2 + |Du|^2) d\mu(x)
$$
on 
 $\mathcal{U}$ and $\mathcal{U}_{\emptyset}$. We then have the following results:

	\begin{corollary}\label{Cor:strictconcave}
		Let $\pi$ and $G$ satisfy \Gzero--\Gfive. If \\
		$(iii)$  $t\in[0,1] \longmapsto \pi(x, y_t ,z_t)$ is strictly concave along all G-segments $(x, y_t, z_t)$, then \\ $(iv)$ $\pmb \Pi(u)$ is strictly concave in $\mathcal{U} \subset W^{1,2}(X,d\mu)$  for all $\mu\ll \mathcal{L}^m$. If \\	
		$(iii)'$ $t\in[0,1] \longmapsto \pi(x, y_t(x) ,z_t(x))$ is strictly concave along all G-segments $(x, y_t(x), z_t(x))$ whose 
		endpoints satisfy $\min\{G(x, y_0(x), z_0(x)),G(x,y_1(x),z_1(x))\} \ge u_{\emptyset}(x)$, then \\
		$(iv)'$ $\pmb \Pi(u)$ is strictly concave in $\mathcal{U}_{\emptyset} \subset W^{1,2}(X,d\mu)$ for all $\mu\ll \mathcal{L}^m$. 
		\end{corollary}

In addition, Theorem $\ref{convex set}$ and Corollary $\ref{Cor:strictconcave}$ together imply strict concavity of principal's profit on a convex space, which guarantees a unique solution to the monopolist's problem.\\

Define $\bar{G}(\bar{x}, \bar{y})=\bar{G}(x,x_0, y,z) := x_0 G(x, y,z)$, where $\bar{x}=(x, x_0)$, $\bar{y}=(y,z)$ and $x_0\in X_0$, where $X_0 \subset (-\infty, 0)$ is an open bounded interval containing $-1$. Hereafter, except in Appendix A, we use $x_0$ to denote a number in $X_0$. For further applications, we need the following non-degeneracy assumption. \\
\begin{itemize}
	\item[\Gsix] $G\in C^2(cl(X\times Y \times Z)
	)$, and $D_{\bar{x},\bar{y}}(\bar{G})(x,-1,y,z)$ has full rank, for each $(x,y,z)\in cl(X\times Y\times Z)$. \\
\end{itemize}

		 Since \Gone\ implies $m \ge n$,  full rank means $D_{\bar{x},\bar{y}}(\bar{G})(x,-1,y,z)$ has rank $n+1$.


\begin{theorem}[Uniform concavity of the principal's objective]\label{maintheorem2}
		Assume $G\in  C^{2}(cl(X\times Y\times Z))$ satisfies \Gzero-{\Gsix}. In case $\bar{z}=+\infty$, assume the homeomorphisms of \Gone\ are uniformly bi-Lipschitz. Then the following statements are equivalent:\\
		$(v)$ Uniformly concavity of $\pi$ along G-segments, i.e., there exists $\lambda>0$, for any $G$-segment $(x, y_t, z_t)$, and any $t\in [0,1]$, 
			\begin{flalign}\label{uniformconvavity}
			\begin{split}
			&\pi(x, y_t(x), z_t(x)) - (1-t)\pi(x, y_0(x), z_0(x)) - t\pi(x, y_1(x), z_1(x)) \\
			\ge & t(1-t)\lambda||(y_1(x)-y_0(x),z_1(x)-z_0(x))||^2_{\R^{n+1}}
			\end{split}
			\end{flalign} \\
		$(vi)$ $\pmb \Pi(u)$ is uniformly concave in $\mathcal{U} \subset W^{1,2}(X,d\mu)$,  uniformly for all $\mu\ll \mathcal{L}^m$. 
\end{theorem}

\begin{proof}
	$(v)\Rightarrow (vi).$ With the same notation as last proof, we want to prove there exists 
	$\tilde{\lambda}>0$, such that $\pmb \Pi(u_t) - (1-t) \pmb \Pi(u_0) - t\pmb \Pi(u_1)\ge t(1-t)\tilde{\lambda}||u_1-u_0||^2_{W^{1,2}(X,d\mu)}$, for any $\mu\ll \mathcal{L}^m$, $u_0, u_1 \in \mathcal{U}$ and $t\in (0,1)$. \\
	Similar to last proof, we have (\ref{Eqn:u_01}) and (\ref{EqnG-segments}). Denote $\Lip(G_x,G)$ the uniform Lipschitz constant of the map $(x,y,z)\in X\times Y\times Z \longmapsto (G_x,G)(x,y,z)$.
	Thus by uniformly concavity of $\pi$ on $G$-segments, there exists 
	$\lambda>0$, such that for every $t \in [0,1]$,
	\begin{align*}
		&\pmb \Pi (u_t)-(1-t)\pmb \Pi (u_0)- t \pmb \Pi(u_1)\\
		&= \int_X \pi(x, y_t(x), z_t(x)) - (1-t)\pi(x, y_0(x), z_0(x)) - t\pi(x, y_1(x), z_1(x)) d \mu (x)\\
		&\ge \int_{X} t(1-t)\lambda||(y_1(x)-y_0(x),z_1(x)-z_0(x))||^2_{\R^{n+1}} d\mu (x)\\
		&\ge \int_{X} t(1-t)\lambda||(Du_1(x)-Du_0(x),u_1(x)-u_0(x))||^2_{\R^{n+1}}/{\Lip}^2(G_x,G) d\mu (x)\\
		&= t(1-t)\frac{\lambda}{\Lip^2(G_x,G)} ||u_1-u_0||_{W^{1,2}(X,d\mu)}.
		\end{align*}
		Thus, $\pmb \Pi $ is uniformly concave in $\mathcal{U}$, with $\tilde{\lambda} = \frac{\lambda}{\Lip^2(G_x,G)} >0$.\\
	
	$(vi)\Rightarrow (v).$ To derive a contradiction, assume $(v)$ fails.  Then for any $\lambda>0$, there exists a $G$-segment $(x^0, y_t(x^0), z_t(x^0))$, and some $\tau \in (0,1)$, such that $\pi(x^0, y_{\tau}(x^0), z_{\tau}(x^0)) -  (1-\tau)\pi(x^0, y_{0}(x^0),z_{0}(x^0))-\tau \pi(x^0, y_{1}(x^0), z_{1}(x^0)) < \tau (1-\tau)\lambda ||(y_{1}(x^0)-y_{0}(x^0), z_{1}(x^0)-z_{0}(x^0))||^2_{\R^{n+1}}$.
	
Take $u_0(x) := G(x, y_{0}(x^0), z_{0}(x^0))$, $u_1(x):= G(x, y_{1}(x^0), z_{1}(x^0))$ and for $t\in (0,1)$, assign $u_{t} := (1-t)u_0 +t u _1$. Then $ u_{t} \in \mathcal{U}$, for $t \in [0,1]$.  From $(\ref{Eqn:u_01})$ we know, $y_i(x)\equiv y_i(x^0)$, $z_i(x)\equiv z_i(x^0)$, for $i=0,1$.  Let $t\in [0,1] \longmapsto (x, y_t(x), z_t(x))$ be the $G$-segment connecting $(x, y_0(x), z_0(x))$ and $(x, y_1(x), z_1(x))$. And combining $(\ref{EqnInverse})$ and $(\ref{$G$-segment})$, we have
		\begin{flalign*}
		(G_x, G) (x, y_{0}(x^0), z_{0}(x^0)) &= (Du_0, u_0)(x),\\
		(G_x, G) (x, y_{1}(x^0), z_{1}(x^0)) &= (Du_1, u_1)(x).\\
		(G_x, G) (x, y_{t}(x), z_{t}(x)) &= (Du_{t}, u_{t})(x).
		\end{flalign*}

		Since $\pi$, $y_{\tau}$ and $z_{\tau}$ are continuous, there exists $\varepsilon >0$, such that for all  $x\in B_{\varepsilon}(x^0)$,
		\begin{flalign*}
		&\pi(x, y_{\tau}(x), z_{\tau}(x)) -  (1-\tau)\pi(x, y_{0}(x^0),z_{0}(x^0))-\tau \pi(x, y_{1}(x^0), z_{1}(x^0)) \\
		&< \tau (1-\tau)\lambda||(y_{1}(x^0)-y_{0}(x^0), z_{1}(x^0)-z_{0}(x^0))||^2_{\R^{n+1}}.
		\end{flalign*} 
		Here we use $B_{\varepsilon}(x^0)$ denote the open ball in $\R^m$ centered at $x^0$ with radius $\varepsilon$.
		Take $d\mu = d\mathcal{L}^m
		\mid _{B_{\varepsilon} (x^0)}/\mathcal{L}
		(B_{\varepsilon}(x^0))$ to be uniform measure on $B_\varepsilon(x^0)$. 
		 By \Gsix,   the map $\bar{y}_G:(x,p,q)\longmapsto(y,z)$, which solves equation (\ref{EqnInverse}), is uniformly Lipschitz on $X\times \R \times\R^m$. Denote  $\Lip(\bar{y}_G)$ its Lipschitz constant.
		
		Thus for such $\tau$, $u_0$, $u_1$ and $\mu$, we have
		\begin{align*}
		&\pmb \Pi (u_{\tau})- (1-\tau)\pmb \Pi(u_0)- \tau \pmb \Pi(u_1)\\
		&= \int_X \pi(x, y_{\tau}(x), z_{\tau}(x)) -  (1-\tau)\pi(x, y_{0}(x^0),z_{0}(x^0))-\tau \pi(x, y_{1}(x^0), z_{1}(x^0)) d\mu(x)\\
		&<  \int_X    \tau(1-\tau)\lambda ||(y_{1}-y_{0}, z_{1}-z_{0})||^2_{\R^{n+1}} d\mu(x)\\
		&\le \tau(1-\tau)\lambda {\Lip}^2(\bar{y}_G)||u_1-u_0||^2_{W^{1,2}(X,d\mu)}.
		\end{align*}
		This contradicts the uniformly concavity of $\pmb \Pi$.
\end{proof}

A similar argument implies the following equivalence. Theorem $\ref{convex set}$ and Corollary $\ref{Cor:concave2}$ together imply that the principal's profit $\pmb \Pi$ is a uniformly concave functional on a convex space, under assumptions \Gzero-\Gsix, $\mu\ll \mathcal{L}^m$, and $(v)'$.  

\begin{corollary}\label{Cor:concave2}
	Under the same assumptions as in Theorem \ref{maintheorem2},  the following are equivalent:\\
	$(v)'$ Uniform concavity of $\pi$  (in the sense of equation (\ref{uniformconvavity})) along G-segments $(x, y_t(x), z_t(x))$ whose 
	endpoints satisfy $\min\{G(x, y_0(x), z_0(x)),G(x,y_1(x),z_1(x))\} \ge u_{\emptyset}(x)$;\\
	$(vi)'$ $\pmb \Pi(u)$ is uniformly concave  in $\mathcal{U}_{\emptyset} \subset W^{1,2}(X,d\mu)$   uniformly
for all $\mu\ll \mathcal{L}^m$. \\
\end{corollary}

The preceding concavity results also have convexity analogs.  Unlike strict concavity,   strict convexity does not
imply uniqueness of the principal's profit-maximizing strategy, though it suggests it should only be attained at 
extreme points of the strategy space $\mathcal U$,  where extreme point needs to be interpreted appropriately.

\begin{remark}[Convexity of principal's objective]\label{remarkmaintheorem1}
If $\pi$ and $G$ satisfy \Gzero--\Gfive,  the equivalences 
$(i) \Leftrightarrow (ii)$ and $(i)' \Leftrightarrow (ii)'$ and implications $(iii) \Rightarrow (iv)$ and $(iii)' \Rightarrow (iv)'$ remain true when all occurences of concavity are replaced by convexity.
Similarly,  the equivalences $(v) \Leftrightarrow (vi)$ and $(v)' \Leftrightarrow (vi)'$ remain true when both
occurences of uniform concavity are replaced by uniform convexity in Theorem \ref{maintheorem2}. 
\end{remark}

We close with several examples,  which are established by computing two derivatives
of $\pi(x,y_t,z_t)$ along an arbitrary $G$-segment $t\in [0,1] \longmapsto (x,y_t,z_t)$.
These computations are tedious but straightforward.

 Assuming \Gsix, we denote $(\bar{G}_{\bar{x}, \bar{y}})^{-1}$ the left inverse of $D_{\bar{x},\bar{y}}(\bar{G})(x,x_0,y,z)$.
We will use Einstein notation for simplifying expressions including summations of vectors, matrices, and general tensors for higher order derivatives. There are essentially three rules of Einstein summation notation, namely: 1. repeated indices are implicitly summed over; 2. each index can appear at most twice in any term; 3.~both sides of an equation must contain the same non-repeated indices. For example, $a_{ij}v_i =\sum_{i}a_{ij}v_i$, $a_{ij}b^{kj}v_k=\sum_{j}\sum_{k}a_{ij}b^{kj}v_k$. We also use comma to separate subscripts: the subscripts before comma represent derivatives with respect to first variable and those after comma represent derivatives with respect to second variable. For instance, for $b=b(x,y)$, $b_{,kl}$ represents second derivatives with respect of $y$ only. And for $G=G(x,y,z)$, where $z\in \R$, $G_{i,jz}$ denotes third order derivatives with respect of $x$, $y$ and $z$, instead of using another comma to separate subscripts corresponding to $y$ and $z$. Starting from now, for subscripts, we use $i,k,j,l, \alpha, \beta$ denoting integers  from either $\{1,...,m\}$ or $\{ 1,..., n\}$, and $\bar{i},\bar{k},\bar{j},\bar{l}$ denoting augmented indices from $\{1,...,m+1\}$ or $ \{1, ..., n+1\}$. For instance, $\pi_{i,}$ denotes first order derivative with respect to $x$ only, $\pi_{,\bar{k}\bar{j}}$ represents Hessian matrix with respect to $\bar{y}$ only, and $\bar{G}_{\bar{i},\bar{k}\bar{j}}$ denotes a third order derivative tensor which can be viewed as taking $\bar{x}$-derivative of $\bar{G}_{,\bar{k}\bar{j}}$.

The following remark reformulates concavity of $\pi$ on $G$-segments using non-positive definiteness of a matrix. This equivalent form provides a simple method to verify concavity condition stated in Theorem $\ref{maintheorem}$. We will apply this matrix form to establish Corollary $\ref{bar{G}^*-Concavity}$ and Example $\ref{general example1} - \ref{general example3}$.

\begin{lemma}[Characterizing concavity of principal's profit in the smooth case]\label{LemmaProfitConcavity}
	When $G \in C^3(cl(X\times Y \times Z))$ satisfies \Gzero-\Gsix \ and  $\pi \in C^2(cl(X\times Y \times Z))$, then differentiating $\pi$ along an arbitrary $G$-segment $t \in[0,1] \longrightarrow (x,y_t,z_t)$ yields
	\begin{equation}\label{pi second}
		\frac{d^2}{dt^2} \pi(x, y_t, z_t) = (\pi_{,\bar{k}\bar{j}}- \pi_{,\bar{l}} \bar{G}^{\bar i,\bar l}\bar{G}_{\bar{i},\bar{k}\bar{j}}) \dot {\bar y}^{\bar k} \dot {\bar y}^{\bar j}
	\end{equation}
where $\bar{G}^{\bar i,\bar l}$ denotes the left inverse of the matrix $\bar{G}_{\bar{i}, \bar{k} }$
and $\dot {\bar y}^{\bar k} = (\frac{d}{dt})\bar y^{\bar k}_t$.
Thus $(i)$ in Theorem $\ref{maintheorem}$  is equivalent to non-positive  definiteness of the quadratic form 
$\pi_{,\bar{k}\bar{j}}- \pi_{,\bar{l}} \bar{G}^{\bar i,\bar l}\bar{G}_{\bar{i},\bar{k}\bar{j}}$
on $T_{\bar y}  (Y\times Z) = \R^{n+1}$, { for each  $(x, \bar y) \in X \times Y \times Z$.} Similarly, Theorem \ref{maintheorem2} $(v)$ is equivalent to uniform negative definiteness of the same form.
\end{lemma}	



%


Before we state another result, we need the following definition, which is a generalized Legendre transform 
(see Moreau \cite{Moreau70}, Kutateladze-Rubinov \cite{KutateladzeRubinov72}, Elster-Nehse \cite{ElsterNehse74}, Balder \cite{Balder77}, Dolecki-Kurcyusz \cite{DoleckiKurcyusz78}, Gangbo-McCann\cite{GangboMcCann96}, Singer\cite{Singer97}, Rubinov\cite{Rubinov00a, Rubinov00b}, and Martínez-Legaz \cite{MartinezLegaz05} for more references).  


\begin{definition}[$\bar{G}$-concavity, $\bar{G}^*$-concavity]\label{(-bar{G})-convexity}
	A function $\phi: cl(X \times X_0) \longrightarrow \R$ is called $\bar{G}$-concave if $\phi = (\phi^{\bar{G}^*})^{\bar{G}}$ and a function $\psi: cl( Y \times Z) \longrightarrow \R$ is called $\bar{G}^*$-concave if $\psi = (\psi^{\bar{G}})^{\bar{G}^*}$, where 
	\begin{equation}\label{(-bar{G})-transform}
	\psi^{\bar{G}}(\bar{x})=\min_{\bar{y} \in cl( Y \times Z)} \bar{G}(\bar{x},\bar{y}) - \psi(\bar{y}) \text{ and } \phi^{\bar{G}^*}(\bar{y})= \min_{\bar{x} \in cl(X \times X_0)} \bar{G}(\bar{x}, \bar{y}) - \phi(\bar{x}).
	\end{equation}
	We say $\psi$ is strictly $\bar{G}^*$-concave, if in addition $\psi^{\bar G} \in C^1(X\times X_0)$. 
\end{definition}

Note that,  apart from an overall sign and the extra variables,
Definition \ref{(-bar{G})-convexity} coincides with a quasilinear version $G(\bar{x},\bar{y},z) = \bar{G}(\bar{x},\bar{y})-z$ of
Definition \ref{$G$-convexity}.
%

The following corollary characterizes the concavity of principal's profit when her utility on one hand is not influenced by the agents' identity, and on the other hand has adequate generality to  encompass a
tangled nonlinear relationship between products and selling prices. It generalizes the convexity result in Figalli-Kim-McCann \cite{FigalliKimMcCann11}, where $G(x,y,z) = b(x,y)-z$ and $\pi(x,y,z) = z-a(y)$.  See Appendix \ref{A:B} for a (partial) converse.


\begin{corollary}[Concavity of principal's objective with her utility not depending on agents' types]\label{bar{G}^*-Concavity}
	If $G \in C^3(cl(X\times Y \times Z)
	)$ satisfies \Gzero-\Gsix,  $\pi \in C^2( cl( Y \times Z)
	)$ is $\bar{G}^*$-concave and $\mu\ll \mathcal{L}^m$, then $\pmb \Pi$ is concave. 
\end{corollary}


For specific non-quasilinear agent preferences,  we use the explicit expression above for the 
desired second derivative to establish the following examples,  which assume the principal is indifferent to
whom she transacts business with and that her preferences depend linearly on payments.  
These examples give conditions under which the principal's program inherits concavity or convexity 
from the agents' price sensitivity.
Although the resulting conditions appear complicated,  they illustrate
the subtle interplay between the preferences of agent and principal for products in the first example,
and between the preferences of the agents for products as opposed to prices in the second. 

\begin{example}[Nonlinear yet homogeneous sensitivity of agents to prices]
\label{general example1}
	Take $\pi(x, y, z) =z- a(y)$, $G(x, y, z) = b(x,y)-f(z)$, satisfying \Gzero-\Gsix,  $G \in C^3(cl(X\times Y \times Z)
	)$, $\pi \in C^2(cl(X\times Y \times Z)
	)$, and assume $\bar{z}<+\infty$.\\
	1. If $f(z)$ is convex [respectively concave] in $cl(Z)$, then $\pmb \Pi(u)$ is concave [respectively convex] for all $\mu\ll \mathcal{L}^m$ if and only if there exist $\varepsilon \ge 0$ such that each $(x,y,z) \in X \times Y\times Z$ and $\xi \in \R^{n}$ satisfy 
	\begin{equation}\label{robert1}
		\pm \Bigg\{a_{kj}(y)-\frac{b_{,kj}(x,y)}{f'(z)}+\Big(\frac{b_{,l}(x,y)}{f'(z)}- a_l(y)\Big)b^{i,l}(x,y) b_{i,kj}(x,y)\Bigg\} \xi^{k}\xi^{j} \ge  \varepsilon \mid \xi\mid ^2.
	\end{equation}
	2. In addition, $\pmb \Pi(u)$ is uniformly concave [respectively uniformly convex] on $W^{1,2}(X,d\mu)$  uniformly for all $\mu\ll \mathcal{L}^m$ if and only if $\pm f''> 0$ and  \eqref{robert1} holds with $\varepsilon >0$.\\
\end{example}

Although 
the next two examples are not completely general, they have the following economic interpretation. The same selling price impacts utility differently for different types of agents. In other words, it models the situation where agents have different sensitivities to the same price.  In Example $\ref{general example2}$,  the principal's utility is linear and depends exclusively on her revenue, which is a simple special case of Example $\ref{general example3}$.

\begin{example}[Inhomogeneous sensitivity of agents to prices, zero cost]\label{general example2}
	Take $\pi(x, y, z) =z$,  $G(x,y,z)$ $= b(x,y)-f(x,z)$,  satisfying \Gzero--\Gsix,  $G \in C^3(cl(X\times Y \times Z)
	)$, $\pi \in C^2(cl(X\times Y \times Z)
	)$, and assume $\bar{z}<+\infty$. Suppose $D_{x,y}b(x,y)$ has full rank for each $(x,y) \in X\times Y$, and denote its left inverse $b^{i,l}(x,y).$\\
	1. If $(x,y,z)\longmapsto h(x,y,z):=f(x,z)-b_{,l}(x,y)b^{i,l}(x,y)f_{i,}(x,z)$ is strictly increasing and convex [respectively concave] with respect to $z$, then $\pmb \Pi(u)$ is concave [respectively convex]  for all $\mu\ll \mathcal{L}^m$ if and only if there exist $\varepsilon \ge 0$ such that each $(x,y) \in X \times Y$ and $\xi \in \R^{n}$ satisfy
\begin{equation}\label{robert2}
\pm \Big\{-b_{,kj}(x,y)+b_{,l}(x,y)b^{i,l}(x,y) b_{i,kj}(x,y)\Big\} \xi^{k}\xi^{j} \ge  \varepsilon \mid \xi\mid ^2.
\end{equation}
	2. In addition, $\pmb \Pi(u)$ is uniformly concave [respectively uniformly convex] on $W^{1,2}(X,d\mu)$ uniformly for all $\mu\ll \mathcal{L}^m$ if and only  if $\pm h_{zz}> 0$ and \eqref{robert2} holds with $\varepsilon >0$.\\

\end{example}

\begin{example}[Inhomogeneous sensitivity of agents to prices]\label{general example3}
	Take $\pi(x, y, z) =z-a(y)$,  $G(x,y,z)= b(x,y)-f(x,z)$,  satisfying \Gzero-\Gsix,  $G \in C^3(cl(X\times Y \times Z)
	)$, $\pi \in C^2(cl(X\times Y \times Z)
	)$, and assume $\bar{z}<+\infty$. Suppose $D_{x,y}b(x,y)$ has full rank for each $(x,y) \in X\times Y$, and $1- (f_{z})^{-1}b_{,\beta}b^{\alpha,\beta}f_{\alpha,z} \ne 0$, for all $(x, y,z) \in X\times Y\times Z$.\\
	1. If $(x,y,z)\longmapsto h(x,y,z):=a_{l}b^{i,l}f_{i,zz}+\frac{[a_{\beta}b^{\alpha,\beta}f_{\alpha,z}-1][b_{,l}b^{i,l}f_{i,zz}-f_{zz}]}{f_{z} -b_{,\beta}b^{\alpha,\beta}f_{\alpha,z}} \ge 0 [\le 0]$ , then $\pmb \Pi(u)$ is concave  [respectively convex] for all $\mu\ll \mathcal{L}^m$ if and only if there exist $\varepsilon \ge 0$ such that each $(x,y,z) \in X \times Y\times Z$ and $\xi \in \R^{n}$ satisfy 
	\begin{equation}\label{robert3}
	\pm \Bigg\{a_{kj} -a_{l}b^{i,l}b_{i,kj}+\frac{1-a_{\beta}b^{\alpha,\beta}f_{\alpha,z}}
	{1- (f_{z})^{-1}b_{,\beta}b^{\alpha,\beta}f_{\alpha,z}}
	\Big[-\frac{b_{,kj}}{f_z}+\frac{b_{,l}}{f_z}b^{i,l} b_{i,kj}\Big] \Bigg\}\xi^{k}\xi^{j} \ge  \varepsilon \mid \xi\mid ^2.
	\end{equation}
	2. If in addition, $\pmb \Pi(u)$ is uniformly concave [respectively uniformly convex] on $W^{1,2}(X,d\mu)$ uniformly for all $\mu\ll \mathcal{L}^m$ if and only if $\pm h>0$ and  \eqref{robert3} holds with $\varepsilon>0$.\\
\end{example}

Example \ref{general example 4} asserts the concavity of monopolist's maximization in the zero-sum setting,
where the agent's utilities are relatively general but the principal's profit is extremely special. In addition, more non-quasilinear examples could be discovered by applying Lemma $\ref{LemmaProfitConcavity}$.

\begin{example}[Zero sum transactions]\label{general example 4}
	Take $\pi(x, y, z) = -G(x,y,z)$, satisfying \Gzero-\Gfive\ and $\mu\ll \mathcal{L}^m$, which means the monopolist's profit in each transaction coincides exactly with the 
agent's loss. From $(\ref{$G$-segment})$, since $G$ is linear on $G$-segments, we know $\pmb \Pi(u)$ is linear.
\end{example}

\bigskip

\section{Proofs}

In this section, we prove
Lemma $\ref{LemmaProfitConcavity}$, Corollary $\ref{bar{G}^*-Concavity}$, 
Example $\ref{general example1}$ and $\ref{general example3}$.

\begin{proof}[Proof of Lemma \ref{LemmaProfitConcavity}]
	For any G-segments $(x, y_t, z_t)$ satisfying equation (\ref{EqnG-segments}) and $\pi \in C^2(cl(X\times Y \times Z)
	)$,
	$t\in[0,1] \longmapsto \pi(x, y_t ,z_t)$ is concave [uniformly concave] if and only if $\frac{d^2}{dt^2} \pi(x, y_t, z_t)\le  0$ $[\le -\lambda ||(\dot{y}_t,\dot{z}_t)||^2_{\R^{n+1}}<0]$, for all $t \in [0,1]$.\\
	On the one hand, since $\frac{d}{dt}\pi(x, y_t, z_t) = \pi_{,\bar{k}} \dot{\bar y}^{\bar{k}}$, taking another derivative with respect of $t$ gives 
	\begin{equation}\label{EqnSecondDiffProfit}
		\frac{d^2}{dt^2} \pi(x, y_t, z_t) = 
\pi_{,\bar{k}\bar{j}}\dot{\bar y}^{\bar{k}} \dot{\bar y}^{\bar{j}}+ \pi_{,\bar{l}}\ddot {\bar y}^{\bar{l}}
	\end{equation}
	On the other hand, taking second derivative with respect of t at both sides of equation (\ref{EqnG-segments}), which is equivalent to $\bar{G}_{\bar{i},}(x, x_0, y_t(x), z_t(x)) = (x_0Du_t, u_t)(x)$, for some fixed $x_0\in X_0$, implies
	\begin{equation}\label{EqnSecondDiffbar{G}}
		\bar{G}_{\bar{i},\bar{k}\bar{j}}\dot{\bar y}^{\bar{k}} \dot {\bar y}^{\bar{j}}+\bar{G}_{\bar{i},\bar{k}} \ddot{\bar y}^{\bar{k}} =0 
	\end{equation}
	Combining equations (\ref{EqnSecondDiffProfit}) with (\ref{EqnSecondDiffbar{G}}) yields
\eqref{pi second}.
For $x \in X$, there is a $G$-segment with any given tangent direction through $\bar y = (y,z) \in Y \times Z$.
Thus, the non-positivity  of $\frac{d^2}{dt^2} \pi(x, y_t, z_t)$ along all G-segments $(x, y_t, z_t)$ is equivalent to non-positive  definiteness of the matrix $(\pi_{,\bar{k}\bar{j}}- \pi_{,\bar{l}} \bar{G}^{\bar i,\bar l}\bar{G}_{\bar{i},\bar{k}\bar{j}})$ on $T_{\bar y} (Y\times Z)=\R^{n+1}$.
	
	In addition, the uniformly concavity of $\pi(x, y_t, z_t)$ along all G-segments $(x, y_t, z_t)$ is equivalent to uniform negative definiteness of $(\pi_{,\bar{k}\bar{j}}- \pi_{,\bar{l}} \bar{G}^{\bar i,\bar l}\bar{G}_{\bar{i},\bar{k}\bar{j}})$ on $\R^{n+1}$.
\end{proof}

\begin{proof}[Proof of Corollary \ref{bar{G}^*-Concavity}]
	According to Lemma \ref{LemmaProfitConcavity}, for concavity, we only need to show non-positive definiteness of $(\pi_{\bar{k}\bar{j}}- \pi_{\bar{l}} \bar{G}^{\bar i,\bar l}\bar{G}_{\bar{i},\bar{k}\bar{j}})$ on $\R^{n+1}$, i.e., for any $\bar{x} = (x,x_0) \in X \times X_0$,  $\bar{y} \in Y\times Z$ and $\xi \in \R^{n+1} $, $\big(\pi_{\bar{k}\bar{j}}(\bar{y})- \pi_{\bar{l}}(\bar{y}) \bar{G}^{\bar i,\bar l}(\bar{x}, \bar{y})\bar{G}_{\bar{i},\bar{k}\bar{j}}(\bar{x}, \bar{y})\big)\xi^{\bar{k}}\xi^{\bar{j}} \le 0$.\\
	For any fixed $\bar{x} = (x, x_0) \in X \times X_0$, $\bar{y} \in Y\times Z$, $\xi \in \R^{n+1}$, there exist $\delta >0$ and $t \in (-\delta, \delta) \longmapsto \bar{y}_t \in Y\times Z$, such that $\bar{y}_t|_{t=0} =\bar{y}$,  $\dot{\bar{y}}|_{t=0} = \xi$ and $\frac{d^2}{dt^2} \bar{G}_{\bar{i}, }(\bar{x}, \bar{y}_t) = 0$. Thus, 
	\begin{equation}\label{EqnSecondDiffbar{G}2}
		0 = \frac{d^2}{dt^2}\bigg|_{t=0}\bar{G}_{\bar{i},}(\bar{x}, \bar{y}_t) = \bar{G}_{\bar{i}, \bar{k}\bar{j}}(\bar{x}, \bar{y}) \xi^{\bar{k}}\xi^{\bar{j}} + \bar{G}_{\bar{i}, \bar{k}}(\bar{x}, \bar{y}) (\ddot{\bar{y}}_t)^{\bar{k}}\Big|_{t=0}
	\end{equation}
	Since $\pi$ is $\bar{G}^*$-concave, we have $\pi(\bar{y}) =\min_{\tilde{x} \in cl(X \times X_0)} \bar{G}(\tilde{x}, \bar{y}) - \phi(\tilde{x})$, for some $\bar{G}$-concave function $\phi$. Since $cl(X \times X_0)$ is compact, for this $\bar{y}$, there exists ${\bar{x}}^* = ({x}^*, {x_0}^*) \in cl(X \times X_0)$, such that 
	 $\pi_{\bar{l}}(\bar{y}) = \bar{G}_{,\bar{l}}({\bar{x}}^*, \bar{y})$ for each $\bar{l} = 1, 2,..., n+1$ and $\pi_{\bar{k}\bar{j}}(\bar{y})\xi^{\bar{k}}\xi^{\bar{j}} \le \bar{G}_{,\bar{k}\bar{j}}({\bar{x}^*}, \bar{y})\xi^{\bar{k}}\xi^{\bar{j}}$  for each $\xi \in \R^{n+1}$. 
Combined with (\ref{EqnSecondDiffbar{G}2}) this yields
	\begin{flalign}\label{Eqn:proofcor}
	\begin{aligned}
		&\big(\pi_{\bar{k}\bar{j}}(\bar{y})- \pi_{\bar{l}}(\bar{y}) \bar{G}^{\bar i,\bar l}(\bar{x}, \bar{y})\bar{G}_{\bar{i},\bar{k}\bar{j}}(\bar{x}, \bar{y})\big)\xi^{\bar{k}}\xi^{\bar{j}} \\
		&\le
		\big(\bar{G}_{,\bar{k}\bar{j}}({\bar{x}}^*, \bar{y})- \bar{G}_{,\bar{l}}({\bar{x}}^*, \bar{y})\bar{G}^{\bar i,\bar l}(\bar{x}, \bar{y})\bar{G}_{\bar{i},\bar{k}\bar{j}}(\bar{x}, \bar{y})\big)\xi^{\bar{k}}\xi^{\bar{j}} \\
		&= \bar{G}_{,\bar{k}\bar{j}}({\bar{x}}^*, \bar{y})\xi^{\bar{k}}\xi^{\bar{j}}+ \bar{G}_{,\bar{l}}({\bar{x}}^*, \bar{y})\cdot (\ddot{\bar{y}}_t)^{\bar{l}}\big|_{t=0}\\
		& = \frac{d^2}{dt^2}\bigg|_{t=0} \bar{G}({\bar{x}}^*, \bar{y}_t)\\
		& = {x_0}^* \cdot\frac{d^2}{dt^2}\bigg|_{t=0} G({x}^*, \bar{y}_t)\\
		&\le 0.
		\end{aligned}
	\end{flalign}
	The last inequality comes from ${x_0}^* \le 0$ and \Gthree.
\end{proof}

\begin{proof}[Proof of Example \ref{general example1}]
	From Lemma \ref{LemmaProfitConcavity}, $\pmb \Pi(u)$ is concave for all $\mu\ll \mathcal{L}^m$ if and only if $(\pi_{,\bar{k}\bar{j}}- \pi_{,\bar{l}} \bar{G}^{\bar i,\bar l}\bar{G}_{\bar{i},\bar{k}\bar{j}})\big|_{x_0=-1}$ is non-positive definite, and uniformly concave uniformly for all $\mu\ll \mathcal{L}^m$ if and only if this matrix is uniform negative definite.\\
	In this example, we have $\pi(x,y,z)= z-a(y)$, $\bar{G}(x,x_0, y,z) = x_{0} G(x,y,z) = x_{0}(b(x,y)-f(z))$. Thus, 
	\begin{flalign*}
		\pi_{,\bar{k}\bar{j}}= \begin{pmatrix}
		-a_{kj} & \mathbf{0}\\
		\mathbf{0} & 0\\
	    \end{pmatrix}, \ \ 
	    \pi_{,\bar{l}}= (-a_{l}, 1), \ \ 
	    \bar{G}_{\bar{i},\bar{l}}\big|_{x_0=-1} = \begin{pmatrix}
	    -b_{i,l} & \mathbf{0}\\
	    b_{,l} & -f'(z)\\
	    \end{pmatrix}.\\	    
	\end{flalign*}
	By \Gfour, $f'(z) >0$ for all $z\in cl(Z)$. By \Gsix, since $\bar{G}_{\bar{i},\bar{l}}\big|_{x_0=-1}$ has the full rank, the matrix $(b_{i,l})$ also has its full rank. Taking $b^{i,l}$ as its left inverse, we have
	\begin{flalign*}
	    \bar{G}^{\bar i,\bar l}\big|_{x_0=-1} = \begin{pmatrix}
	    -b^{i,l} & \mathbf{0}\\
	    -\frac{b_{,l}b^{i,l}}{f'(z)} & \frac{1}{-f'(z)}\\
	    \end{pmatrix}, \ \ 
	    \bar{G}_{\bar{i},\bar{k}\bar{j}}\big|_{x_0=-1} =\begin{pmatrix}
	    -b_{i,k\bar{j}}& \mathbf{0}  \\
	    b_{,k\bar{j}} & (-f'(z))_{\bar{j}}\\
	    \end{pmatrix}.
	\end{flalign*}
	Therefore, 
	\begin{flalign*}
		(\pi_{,\bar{k}\bar{j}} - \pi_{,\bar{l}}\bar{G}^{\bar i,\bar l} \bar{G}_{\bar{i},\bar{k}\bar{j}})\big|_{x_0=-1} = &\begin{pmatrix}
		-a_{kj} &\mathbf{0} \\
		\mathbf{0} &  0\\
		\end{pmatrix} - \begin{pmatrix}
		(-a_{l}b^{i,l}+\frac{b_{,l}}{f'(z)}b^{i,l}) b_{i,k\bar{j}}-\frac{b_{,k\bar{j}}}{f'(z)},\frac{(f'(z))_{\bar{j}}}{f'(z)}
		\end{pmatrix}\\
		&= - \begin{pmatrix}
		a_{kj}+(-a_{l}+\frac{b_{,l}}{f'(z)})b^{i,l} b_{i,kj}-\frac{b_{,kj}}{f'(z)} & \mathbf{0}\\
		\mathbf{0} & \frac{f''(z)}{f'(z)}\\
		\end{pmatrix}.
	\end{flalign*}
	Since \Gfour\ and $f$ is convex, we have $f'(z) >0$ and $f''(z)\ge 0$, for all $z\in cl(Z)$. Thus, $\pi_{,\bar{k}\bar{j}}-\pi_{,\bar{l}}\bar{G}^{\bar i,\bar l}\bar{G}_{\bar{i},\bar{k}\bar{j}}$ is non-positive definite if and only if $a_{kj}+(-a_{l}+\frac{b_{,l}}{f'(z)})b^{i,l} b_{i,kj}-\frac{b_{,kj}}{f'(z)}$ is non-negative definite, i.e., 
	there exist $\varepsilon \ge 0$ such that each $(x,y,z) \in X \times Y\times Z$ and $\xi \in \R^{n}$ satisfy 
	$$ \Bigg\{a_{kj}(y)-\frac{b_{,kj}(x,y)}{f'(z)}+\Big(\frac{b_{,l}(x,y)}{f'(z)}- a_l(y)\Big)b^{i,l}(x,y) b_{i,kj}(x,y)\Bigg\} \xi^{k}\xi^{j} \ge  \varepsilon \mid \xi\mid ^2.
	$$\\
	In addition, $\pi_{,\bar{k}\bar{j}}-\pi_{,\bar{l}}\bar{G}^{\bar i,\bar l}\bar{G}_{\bar{i},\bar{k}\bar{j}}$ is uniform negative definite if and only if $f''>0$ and $\varepsilon>0$, which is equivalent to that $\pmb \Pi(u)$ is uniformly concave uniformly for all $\mu\ll \mathcal{L}^m$. Similarly, one can show equivalent conditions for $\pmb \Pi(u)$ being convex or uniformly convex.
\end{proof}

\begin{proof}[Proof of Example \ref{general example3}]
	Similar to the proof of Example \ref{general example1},  $\pmb \Pi(u)$ is concave for all $\mu\ll \mathcal{L}^m$ if and only if $(\pi_{,\bar{k}\bar{j}}-\pi_{,\bar{l}}\bar{G}^{\bar i,\bar l}\bar{G}_{\bar{i},\bar{k}\bar{j}})$ is non-positive definite, and  uniformly concave uniformly for all $\mu\ll \mathcal{L}^m$ if and only if this tensor is uniform negative definite.\\
	Since $D_{x,y}b(x,y)$ has full rank for each $(x,y) \in X\times Y$, and $1- (f_{z})^{-1}b_{,\beta}b^{\alpha,\beta}f_{\alpha,z} \ne 0$, for all $(x, y,z) \in X\times Y\times Z$, for $\pi(x,y,z) = z-a(y)$, $\bar{G}(x,x_0, y, z) = x_0(b(x,y)-f(x,z))$, we have 
	\begin{flalign*}
	&(\pi_{,\bar{k}\bar{j}} - \pi_{,\bar{l}}\bar{G}^{\bar i,\bar l} \bar{G}_{\bar{i},\bar{k}\bar{j}}) \\
	&=- \begin{pmatrix}
	 \begin{split}
	 &\hspace{2cm}a_{kj} -a_{l}b^{i,l}b_{i,kj}\\
	 &+\frac{[a_{\beta}b^{\alpha,\beta}f_{\alpha,z}-1][b_{,kj}-b_{,l}b^{i,l} b_{i,kj}]}{f_{z} -b_{,\beta}b^{\alpha,\beta}f_{\alpha,z}}
	 \end{split} & \begin{split}
	 \mathbf{0}
	 \end{split}\\
	\begin{split}
	\mathbf{0}
	\end{split} & \begin{split}
	&\hspace{2.7cm}a_{l}b^{i,l}f_{i,zz}\\
	&+\frac{[a_{\beta}b^{\alpha,\beta}f_{\alpha,z}-1][b_{,l}b^{i,l}f_{i,zz}-f_{zz}]}{f_{z} -b_{,\beta}b^{\alpha,\beta}f_{\alpha,z}}
	\end{split}\\
	\end{pmatrix}.
	\end{flalign*}
	Since $h(x,y,z)\ge 0 $ , then $(\pi_{,\bar{k}\bar{j}}-\pi_{,\bar{l}}\bar{G}^{\bar i,\bar l}\bar{G}_{\bar{i},\bar{k}\bar{j}})$ is non-positive definite if and only if there exist $\varepsilon \ge 0$ such that each $(x,y,z) \in X \times Y\times Z$ and $\xi \in \R^{n}$ satisfy 
	$$ \Bigg\{a_{kj} -a_{l}b^{i,l}b_{i,kj}+\frac{[a_{\beta}b^{\alpha,\beta}f_{\alpha,z}-1][b_{,kj}-b_{,l}b^{i,l} b_{i,kj}]}{f_{z} -b_{,\beta}b^{\alpha,\beta}f_{\alpha,z}}\Bigg\} \xi^{k}\xi^{j} \ge  \varepsilon \mid \xi\mid ^2.
	$$\\
	In addition, $\pmb \Pi(u)$ is uniformly concave uniformly for all $\mu\ll \mathcal{L}^m$ if and only if $h>0$ and $\varepsilon>0$.
\end{proof}

\bigskip

\appendix

\section{A fourth-order differential re-expression of \Gthree}

In our convexity argument, hypothesis \Gthree\ plays a crucial role. In this section, we 
localize this hypothesis using differential calculus.  The resulting expression
shows it to be a direct analog of the non-negative cross-curvature (B3) from \cite{FigalliKimMcCann11},
which in turn was inspired by \cite{MaTrudingerWang05}. \\

In this section only, we assume the dimensions of spaces $X$ and $Y$ are equal, i.e., $m=n$. Before the statement, we need to extend the twist and convex range hypotheses \Gone\ and \Gtwo\  
to the function $H$ in place of $G$. This is equivalent to assuming
:

\begin{itemize}
	\item[\Gseven] For each $(y, z)\in cl( Y \times Z)$ the map $x \in X \longmapsto \frac{G_y}{G_z}(\cdot, y,z)$ is one-to-one;\\
	\item[\Geight] its range $X_{(y,z)} := \frac{G_y}{G_z}(X,y,z) \subset \R^n$ is convex.\\
\end{itemize}

For any vectors $p, w \in \R^n$, we denote $p\parallel w$ if $p$ and $w$ are parallel.

\begin{proposition}
	Assume \Gzero-\Gtwo\ and \Gfour-\Geight. If, in addition, $G\in C^4(cl(X\times Y \times Z)
	)$,  then the following statements are equivalent:\\
	\begin{enumerate}[(i)]
		\item \Gthree.
		
		\item 	 For any given  $x_0, x_1\in X$, any curve $(y_t, z_t) \in cl( Y \times Z)$ connecting $(y_0,z_0)$ and $(y_1, z_1)$, we have 
		\begin{equation*}
		\frac{\partial^2}{\partial s^2 }\Biggl(\frac1{G_z(x_s, y_t, z_t)}\frac{\partial^2}{\partial t^2} G(x_s,y_t,z_t) \Biggr)\Bigg|_{t=t_0}\le 0,
		\end{equation*}
		whenever  { $s\in [0,1] \longmapsto \frac{G_y}{G_z}(x_s, y_{t_0}, z_{t_0})$} forms an affinely parametrized line segment for some $t_0 \in [0,1]$.\\
		
		\item For any given curve $x_s\in X$ connecting $x_0$ and $x_1$,  any $(y_0, z_0),  (y_1, z_1) \in cl( Y \times Z) $, we have 
		\begin{equation*}
		\frac{\partial^2}{\partial s^2 }\Biggl(\frac{1}{G_z(x_s, y_t, z_t)}\frac{\partial^2}{\partial t^2} G(x_s,y_t,z_t) \Biggr)\Bigg|_{s=s_0}\le 0,
		\end{equation*}
		whenever $t\in [0,1] \longmapsto (G_x, G)(x_{s_0}, y_t, z_t)$  forms an affinely parametrized line segment for some $s_0\in [0,1]$.\\
		
	\end{enumerate} 
\end{proposition}

\begin{proof}
	$(i)\Rightarrow (ii).$ Suppose for some $t_0 \in [0,1]$, $s \in [0,1] \longmapsto \frac{G_y}{G_z}(x_s, y_{t_0}, z_{t_0})$ forms an affinely parametrized line segment.\\
	For any fixed $s_0\in [0,1]$, consider $x_{s_0}\in X$, there is a $G$-segment $(x_{s_0}, y_t^{s_0}, z_t^{s_0})$ passing through $(x_{s_0}, y_{t_0}, z_{t_0})$ at $t=t_0$ with the same tangent vector as $(x_{s_0}, y_t, z_t)$ at $t=t_0$, i.e., there exists
another curve $(y_t^{s_0}, z_t^{s_0}) \in cl( Y \times Z)$, such that $(y_t^{s_0},z_t^{s_0})\mid _{t=t_0} = (y_t,z_t)\mid _{t=t_0}$,  $(\dot{y_t}^{s_0},\dot{z_t}^{s_0})\mid _{t=t_0} \parallel (\dot{y}_t,\dot{z}_t)\mid _{t=t_0}$, and $(G_x, G)(x_{s_0},y_t^{s_0},z_t^{s_0}) = (1-t)(G_x, G)(x_{s_0},y_0^{s_0},z_0^{s_0})+t (G_x, G)(x_{s_0},y_1^{s_0},z_1^{s_0})$. \\
	Computing the fourth mixed derivative yields
	\begin{flalign*}
	&\frac{\partial^2}{\partial s^2 }\Biggl(\frac{1}{G_z(x_s, y_t, z_t)}\frac{\partial^2}{\partial t^2} G(x_s,y_t,z_t) \Biggr)\\
	=&\frac{\partial^2}{\partial s^2}\Biggl(\frac{1}{G_z}\Biggr) \frac{\partial^2}{\partial t^2} G + 2 \frac{\partial}{\partial s} \Biggl(\frac{1}{G_z}\Biggr) \frac{\partial^3}{\partial s \partial t^2} G + \frac{1}{G_z} \frac{\partial^4}{\partial s^2 \partial t^2} G  \\
	=& [-(G_z)^{-2}G_{i,z} \ddot{x_s}^{i} - (G_z)^{-2} G_{ij,z}\dot{x_s}^i \dot{x_s}^j+ 2(G_z)^{-3} G_{i,z} G_{j,z} \dot{x_s}^i \dot{x_s}^j] [G_{,k}\ddot{y_t}^{k}+G_z\ddot{z}_t+G_{,kl}\dot{y_t}^k \dot{y_t}^l \\
	&\hspace{0.2cm}+ 2G_{,kz} \dot{y_t}^k \dot{z}_t +G_{zz} (\dot{z}_t)^2]\\
	&+ 2[-(G_z)^{-2}G_{i,z}\dot{x_s}^i] [G_{j,k}\dot{x_s}^j\ddot{y_t}^{k}+G_{j,z}\dot{x_s}^j\ddot{z}_t+G_{j,kl}\dot{x_s}^j\dot{y_t}^k \dot{y_t}^l + 2G_{j,kz}\dot{x_s}^j \dot{y_t}^k \dot{z}_t +G_{j,zz} \dot{x_s}^j(\dot{z}_t)^2]\\
	&+(G_z)^{-1}[G_{i,k}\ddot{x_s}^i\ddot{y_t}^{k} +G_{ij,k}\dot{x_s}^i\dot{x_s}^j\ddot{y_t}^{k} +G_{i,z}\ddot{x_s}^i\ddot{z}_t +G_{ij,z}\dot{x_s}^i\dot{x_s}^j\ddot{z}_t +G_{i,kl}\ddot{x_s}^i\dot{y_t}^k \dot{y_t}^l +G_{ij,kl}\dot{x_s}^i\dot{x_s}^j\dot{y_t}^k \dot{y_t}^l \\
	&\hspace{0.2cm}+2G_{i,kz}\ddot{x_s}^i\dot{y_t}^k\dot{z}_t +2G_{ij,kz}\dot{x_s}^i\dot{x_s}^j\dot{y_t}^k\dot{z}_t +G_{i,zz}\ddot{x_s}^i(\dot{z}_t)^2 +G_{ij,zz}\dot{x_s}^i\dot{x_s}^j(\dot{z}_t)^2]\\ 
	=& [((G_z)^{-1}G_{i,k}-(G_z)^{-2}G_{i,z}G_{,k})\ddot{x_s}^i +((G_z)^{-1}G_{ij,k}- (G_z)^{-2} G_{,k}G_{ij,z} -2(G_z)^{-2}G_{i,z}G_{j,k} \\
	& +2(G_z)^{-3}G_{,k}G_{i,z} G_{j,z}) \dot{x_s}^i \dot{x_s}^j]\ddot{y_t}^k \\
	&+ [(G_z)^{-1}G_{i,kl}-(G_z)^{-2}G_{i,z}G_{,kl}]\ddot{x_s}^i\dot{y_t}^k \dot{y_t}^l \\
	&+ [(G_z)^{-1}G_{ij,kl}-(G_z)^{-2} G_{ij,z}G_{,kl}+ 2(G_z)^{-3} G_{i,z} G_{j,z}G_{,kl} -2(G_z)^{-2}G_{i,z}G_{j,kl}]\dot{x_s}^i \dot{x_s}^j\dot{y_t}^k \dot{y_t}^l\\
	&+[2(G_z)^{-1}G_{i,kz}-2(G_z)^{-2}G_{i,z}G_{,kz}] \ddot{x_s}^{i}\dot{y_t}^k\dot{z}_t\\
	&+[2(G_z)^{-1}G_{ij,kz}-2(G_z)^{-2} G_{ij,z}G_{,kz}+ 4(G_z)^{-3}G_{i,z}G_{j,z}G_{,kz} -4(G_z)^{-2}G_{i,z}G_{j,kz}]\dot{x_s}^i\dot{x_s}^j\dot{y_t}^k\dot{z}_t\\
	&+[(G_z)^{-1}G_{i,zz}-(G_z)^{-2}G_{i,z}G_{zz}] \ddot{x_s}^{i}(\dot{z}_t)^2\\
	&+[(G_z)^{-1}G_{ij,zz} -(G_z)^{-2} G_{ij,z}G_{zz} +2(G_z)^{-3} G_{i,z} G_{j,z}G_{zz} -2(G_z)^{-2}G_{i,z}G_{j,zz}]\dot{x_s}^i \dot{x_s}^j(\dot{z}_t)^2.\\
	\end{flalign*}
	The coefficient of $\ddot{y_t}^k$ vanishes when this expression is evaluated at $t =  t_0$, due to the assumption that $s \in [0,1] \longmapsto \frac{G_y}{G_z}(x_s, y_{t_0}, z_{t_0})$ forms an affinely parametrized line segment, which implies
	
	\begin{flalign*}
	0=& \frac{\partial^2}{\partial s^2} \frac{G_y}{G_z}(x_s, y_{t_0}, z_{t_0})\\
	=&[((G_z)^{-1}G_{i,y}-(G_z)^{-2}G_{i,z}G_{,y})\ddot{x}^i +((G_z)^{-1}G_{ij,y}- (G_z)^{-2} G_{,y}G_{ij,z} -2(G_z)^{-2}G_{i,z}G_{j,y} \\
	& +2(G_z)^{-3}G_{,y}G_{i,z} G_{j,z}) \dot{x}^i \dot{x}^j], \mbox{ for all } s\in [0,1].\\
	\end{flalign*}
	
	Since $(\dot{y_t}^{s_0},\dot{z_t}^{s_0})|_{t=t_0} \parallel (\dot{y}_t,\dot{z}_t)|_{t=t_0}$, there exists some constant $C_1>0$, such that $(\dot{y}_t,\dot{z}_t)|_{t=t_0} =C_1(\dot{y_t}^{s_0},\dot{z_t}^{s_0})|_{t=t_0} $. Moreover, since  $ (y_t,z_t)\mid _{t=t_0}=(y_t^{s_0},z_t^{s_0})\mid _{t=t_0} $, we have
	\begin{flalign*}
	\frac{\partial^2}{\partial s^2 }\Biggl(\frac{1}{G_z(x_s, y_t, z_t)}\frac{\partial^2}{\partial t^2} G(x_s,y_t,z_t) \Biggr) \Bigg|_{t=t_0} = C_1^2	\frac{\partial^2}{\partial s^2 }\Biggl(\frac{1}{G_z(x_s, y_t^{s_0}, z_t^{s_0})}\frac{\partial^2}{\partial t^2} G(x_s,y_t^{s_0},z_t^{s_0}) \Biggr) \Bigg|_{t=t_0} .
	\end{flalign*}
	Denote $g(s):=\frac{\partial^2}{\partial t^2}|_{t=t_0} G(x_s,y_t^{s_0},z_t^{s_0})$, for $s\in[0,1]$. Since $(x_{s_0},y_t^{s_0},z_t^{s_0})$ is a $G$-segment, by \Gthree, we have $g(s) \ge 0$, for all $s \in [0,1]$.
	By definition of $(y_t^{s_0},z_t^{s_0})$, it is clear that $g(s_0) =0$.  The first- and second-order
conditions for an interior minimum then give $g'(s_0) =0 \le g''(s_0)$; (in fact $g'(s_0)=0$ also follows
directly from the definition of a $G$-segment).\\
	By the assumption \Gfour,  we have $G_z <0$, thus, 
	\begin{flalign*}
	&\frac{\partial^2}{\partial s^2 }\Biggl(\frac{1}{G_z(x_s, y_t, z_t)}\frac{\partial^2}{\partial t^2} G(x_s,y_t,z_t) \Biggr) \Bigg|_{(s,t)=(s_0,t_0)}\\
	=& C_1^2 	\frac{\partial^2}{\partial s^2 }\Biggl(\frac{1}{G_z(x_s, y_t^{s_0}, z_t^{s_0})}\frac{\partial^2}{\partial t^2} G(x_s,y_t^{s_0},z_t^{s_0}) \Biggr) \Bigg|_{(s,t)=(s_0,t_0)} \\
	=&C_1^2\frac{\partial^2}{\partial s^2}\Bigg|_{(s,t)=(s_0,t_0)} \Biggl(\frac{1}{G_z}(x_s,y_t^{s_0},z_t^{s_0})\Biggr) g(s_0)+ 2C_1^2 \frac{\partial}{\partial s}\Bigg|_{(s,t)=(s_0,t_0)} \Biggl(\frac{1}{G_z(x_s,y_t^{s_0},z_t^{s_0})}\Biggr) g'(s_0)\\
	& + \frac{C_1^2}{G_z(x_s,y_t^{s_0},z_t^{s_0})} g''(s_0)  \\
	\le & 0.
	\end{flalign*}
	\\
	
	%
	$(ii)\Rightarrow (i).$ For any fixed $x_0 \in X$ and $G$-segment $(x_0, y_t, z_t)$,  we need to show $\frac{\partial^2}{\partial t^2}G(x_1, y_t, z_t) \ge 0$,  for all $t \in [0,1]$ and $x_1 \in X$.\\
	
	For any fixed $t_0 \in [0,1]$ and $x_1 \in X$, define $x_s$ as the solution $\hat{x}$ to the equation
	
	\begin{equation}
	\frac{G_y}{G_z}(\hat{x}, y_{t_0},z_{t_0}) = (1-s) \frac{G_y}{G_z}(x_0, y_{t_0},z_{t_0}) +s \frac{G_y}{G_z}(x_1, y_{t_0},z_{t_0}).
	\end{equation}
	
	By \Gseven\ and \Geight, $x_s$ is uniquely determined for each $s\in (0,1)$. In addition, $x_0$ and $x_1$ satisfy the above equation for $s =0$ and $s=1$, respectively.  \\
	
	Define $g(s):=
\frac{1}{G_z(x_s,y_t,z_t)}\frac{\partial^2}{\partial t^2}G(x_s, y_t, z_t)
\Big|_{t=t_0}$ for $s\in [0,1]$.\\
	
	Then $g(0) =0 = g'(0)$ from the two conditions defining a $G$-segment. \\
	
	In our setting,  $s\in [0,1] \longmapsto \frac{G_y}{G_z}(x_s, y_{t_0}, z_{t_0})$ forms an affinely parametrized line segment, thus
	$0\ge \frac{\partial^2}{\partial s^2 }\Bigl(\frac{1}{G_z(x_s, y_t, z_t)}\frac{\partial^2}{\partial t^2} G(x_s,y_t,z_t) \Bigr)\Big|_{t=t_0} = g''(s)$ for all $ s \in [0,1]$ by hypothesis $(ii)$.\\
	
	Hence $g$ is concave in $[0,1]$, and $g(0)=0$ is a critical point, thus $g(1)\le 0$. Since $G_z<0$ this implies $\frac{\partial^2}{\partial t^2}\Big|_{t=t_0}G(x_1, y_t, z_t) \ge 0$ for any $t_0 \in [0,1]$ and $x_1 \in X$,
as desired.\\
	\\
	
	$(i)\Rightarrow (iii).$
	For any fixed $s_0\in [0,1]$, suppose $t\in [0,1] \longmapsto (G_x, G)(x_{s_0}, y_t, z_t)$  forms an affinely parametrized line segment. For any fixed $t_0 \in [0,1]$, define $g(s):=\Big(\frac{1}{G_z(x_s,y_t,z_t)}\frac{\partial^2}{\partial t^2}G(x_s, y_t, z_t)\Big)\Big|_{t=t_0}$, for all $s \in [0,1]$. By \Gthree--\Gfour, we know $g(s)\le 0$, for all $s\in [0,1]$. By the definition of $(y_t, z_t)$, we have $g(s_0)=g'(s_0) =0$. Thus $g''(s_0)\le 0$.\\
	\\
	
	$(iii)\Rightarrow (i).$ For any fixed $x_0 \in X$, suppose $(x_0, y_t^{0}, z_t^{0})$ is a $G$-segment, then we need to show $\frac{\partial^2}{\partial t^2}G(x_1, y_t^{0}, z_t^{0}) \ge 0$,  for all $t \in [0,1], x_1 \in X$.\\
	
	For any fixed $t_0 \in [0,1]$, $x_1 \in X$, define $x_s$ as the solution $\hat{x}$ of equation
	
	\begin{equation}
	\frac{G_y}{G_z}(\hat{x}, y_{t_0}^{0},z_{t_0}^{0}) = (1-s) \frac{G_y}{G_z}(x_0, y_{t_0}^{0},z_{t_0}^{0}) +s \frac{G_y}{G_z}(x_1, y_{t_0}^{0},z_{t_0}^{0}).
	\end{equation}
	
	By \Gseven\ and \Geight, $x_s$ is uniquely determined for each $s\in (0,1)$. In addition, $x_0$ and $x_1$ satisfy the above equation for $s =0$ and $s=1$ respectively.  \\
	
	Define $g(s):=\frac{1}{G_z(x_s,y_t^{0},z_t^{0})}\frac{\partial^2}{\partial t^2}G(x_s, y_t^{0}, z_t^{0})\Big|_{t=t_0}$, for $s\in [0,1]$.\\
	
	Then $g(0) =g'(0)  =0 $ by the two conditions defining a $G$-segment. \\
	
	For any fixed $s_0 \in [0,1]$, there is a $G$-segment $(x_{s_0}, y_t^{s_0}, z_t^{s_0})$ passing through $(x_{s_0}, y_{t_0}^{0}, z_{t_0}^{0})$ at $t=t_0$ with the same tangent vector as $(x_{s_0}, y_t^{0}, z_t^{0})$ at $t=t_0$, i.e., there exists 
	 another curve $(y_t^{s_0}, z_t^{s_0}) \in cl( Y \times Z)$ and some constant $C_2>0$, such that $(y_t^{s_0},z_t^{s_0})\mid _{t=t_0} = (y_t^{0},z_t^{0})\mid _{t=t_0}$,  $(\dot{y_t}^{s_0},\dot{z_t}^{s_0})\mid _{t=t_0} = \frac{1}{C_2} (\dot{y}_t^{0},\dot{z}_t^{0})\mid _{t=t_0}  $, and $(G_x, G)(x_{s_0},y_t^{s_0},z_t^{s_0}) = (1-t)(G_x, G)(x_{s_0},y_0^{s_0},z_0^{s_0})+t (G_x, G)(x_{s_0},y_1^{s_0},z_1^{s_0})$. \\
	Computing the mixed fourth derivative yields
	\begin{flalign*}
	&\frac{\partial^2}{\partial s^2 }\Biggl(\frac{1}{G_z(x_s, y_t^{0}, z_t^{0})}\frac{\partial^2}{\partial t^2} G(x_s,y_t^{0},z_t^{0}) \Biggr) \Bigg|_{(s,t)=(s_0, t_0)} \\
	=& C_2^2	\frac{\partial^2}{\partial s^2 }\Biggl(\frac{1}{G_z(x_s, y_t^{s_0}, z_t^{s_0})}\frac{\partial^2}{\partial t^2} G(x_s,y_t^{s_0},z_t^{s_0}) \Biggr) \Bigg|_{(s,t)=(s_0, t_0)} ,
	\end{flalign*}
	where the equality is derived from the condition that $s \in [0,1] \longmapsto \frac{G_y}{G_z}(x_s, y_{t_0}, z_{t_0})$ forms an affinely parametrized line segment,  $(y_t^{s_0},z_t^{s_0})\mid _{t=t_0} = (y_t,z_t)\mid _{t=t_0}$ and $ (\dot{y}_t^{0},\dot{z}_t^{0})\mid _{t=t_0} = C_2(\dot{y_t}^{s_0},\dot{z_t}^{s_0})\mid _{t=t_0} $. Moreover, the latter expression is non-positive by assumption $(iii)$.\\
	
	Thus $g''(s_0)\le 0$ for all $s_0 \in [0,1]$. Since $g$ is concave in $[0,1]$, and $g(0)=0$ is a critical point, we have $g(1)\le 0$. Thus $G_z<0$ implies $\frac{\partial^2}{\partial t^2}\Big|_{t=t_0}G(x_1, y_t^{0}, z_t^{0}) \ge 0$ for all $t_0 \in [0,1]$ and $x_1 \in X$, as desired. 
\end{proof}

		For strictly concavity of the profit functional, one might need a strict version of hypothesis \Gthree:
		
		\Gthree$_{s}$ For each $x,x_0 \in X$ and $x\ne x_0$, assume $t \in [0,1] \longmapsto G(x, y_t, z_t)$ is strictly convex along all $G$-segments $(x_0, y_t, z_t)$ defined in (\ref{$G$-segment}).\\
	
	\begin{remark}\label{(C5)_s and (C5)_u}
		Strict inequality in $(ii)$ [or $(iii)$] implies \Gthree$_{s}$ but the reverse is not necessarily true, i.e. \Gthree$_{s}$ is intermediate in strength between \Gthree\ and strict inequality version of $(ii)$ [or $(iii)$]. Besides, strict inequality versions of $(ii)$ and $(iii)$ are equivalent, and denoted by $\Gthree_{u}$.

Note inequality \eqref{foc} below and its strict and uniform versions \Gthree$_{s}$ and \Gthree$_{u}$
precisely generalize of the analogous hypotheses $(B3)$, $(B3)_{s}$ and $(B3)_{u}$ from the quasilinear case 
\cite{FigalliKimMcCann11}.
	\end{remark}
	\begin{proof}
		We only show strict inequality of $(ii)$ implies that of $(iii)$ here, since the other direction is similar.\\
		For any fixed $s_0\in [0,1]$, suppose $t\in [0,1] \longmapsto (G_x, G)(x_{s_0}, y_t, z_t)$  forms an affinely parametrized line segment. For any fixed $t_0 \in [0,1]$, define ${x}_s^{t_0}$ as a solution to the equation
		$\frac{G_y}{G_z}({x}_{s}^{t_0}, y_{t_0},z_{t_0})$ $= (1-s) \frac{G_y}{G_z}({x}_{0}^{t_0}, y_{t_0},z_{t_0}) +s \frac{G_y}{G_z}({x}_{1}^{t_0}, y_{t_0},z_{t_0})$, with initial conditions ${x}_s^{t_0} |_{s=s_0}= x_{s_0}$ and $\dot{x}_s^{t_0}|_{s=s_0} = C_1 \dot{x}_s|_{s=s_0}$, for some constant $C_1 >0$.	
		Thus, by strict inequality of $(ii)$, we have 
		\begin{flalign*}
			0>&\frac{\partial^2}{\partial s^2 }\Big(\frac1{G_z(x_s^{t_0}, y_t, z_t)}\frac{\partial^2}{\partial t^2} G(x_s^{t_0},y_t,z_t) \Big)\Big|_{(s,t)=(s_0, t_0)}\\
			=& \frac{\partial^2}{\partial s^2 }\Big(\frac1{G_z(x_s^{t_0}, y_t, z_t)}\Big)\frac{\partial^2}{\partial t^2} G(x_s^{t_0},y_t,z_t) \Big|_{(s,t)=(s_0, t_0)}+ \frac{\partial}{\partial s }\Big(\frac1{G_z(x_s^{t_0}, y_t, z_t)}\Big)\frac{\partial^3}{\partial s \partial t^2} G(x_s^{t_0},y_t,z_t) \Big|_{(s,t)=(s_0, t_0)}\\
			&+ \frac1{G_z(x_s^{t_0}, y_t, z_t)}\frac{\partial^4}{\partial s^2 \partial t^2} G(x_s^{t_0},y_t,z_t) \Big|_{(s,t)=(s_0, t_0)}\\
			=& 
			-\frac{G_{x,z}(x_s^{t_0},y_t,z_t)}{G_z^2(x_s^{t_0}, y_t, z_t)}\frac{\partial^2}{\partial t^2} G_x(x_s^{t_0},y_t,z_t) (\dot{x}_s^{t_0})^2 \Big|_{(s,t)=(s_0, t_0)}\\
			&+ \frac1{G_z(x_s^{t_0}, y_t, z_t)}\frac{\partial^2}{\partial t^2} G_{xx}(x_s^{t_0},y_t,z_t)(\dot{x}_s^{t_0})^2  \Big|_{(s,t)=(s_0, t_0)}\\
			=&C_1^2\Big[-\frac{G_{x,z}(x_{s},y_t,z_t)}{G_z^2(x_{s}, y_t, z_t)}\frac{\partial^2}{\partial t^2} G_x(x_{s},y_t,z_t) (\dot{x}_s)^2 \Big|_{(s,t)=(s_0, t_0)}\\
			&+ \frac1{G_z(x_{s}, y_t, z_t)}\frac{\partial^2}{\partial t^2} G_{xx}(x_{s},y_t,z_t)(\dot{x}_s)^2  \Big|_{(s,t)=(s_0, t_0)}\Big]\\
			=&C_1^2\frac{\partial^2}{\partial s^2 }\Big(\frac1{G_z(x_s, y_t, z_t)}\frac{\partial^2}{\partial t^2} G(x_s,y_t,z_t) \Big)\Big|_{(s,t)=(s_0, t_0)}.\\
		\end{flalign*}
		Here we use the initial condition ${x}_s^{t_0} |_{s=s_0}= x_{s_0}$ and $\dot{x}_s^{t_0}|_{s=s_0} = C_1 \dot{x}_s|_{s=s_0}$.  Besides, since $(x_{s_0}, y_t, z_t)$ forms a $G$-segment, therefore we have $\frac{\partial^2}{\partial t^2} G(x_s^{t_0},y_t,z_t) \Big|_{(s,t)=(s_0, t_0)}=\frac{\partial^2}{\partial t^2} G(x_s,y_t,z_t) \Big|_{(s,t)=(s_0, t_0)}$ $=0$ and $\frac{\partial^2}{\partial t^2} G_x(x_s^{t_0},y_t,z_t) \Big|_{(s,t)=(s_0, t_0)}$ $=\frac{\partial^2}{\partial t^2} G_x(x_s,y_t,z_t) \Big|_{(s,t)=(s_0, t_0)}=0$. \\
		From the above inequality and $C_1 >0$, one has $\frac{\partial^2}{\partial s^2 }\Big(\frac1{G_z(x_s, y_t, z_t)}\frac{\partial^2}{\partial t^2} G(x_s,y_t,z_t) \Big)\Big|_{(s,t)=(s_0, t_0)}<0$, whenever $\dot{x_s}|_{s=s_0}$ and $(\dot{y_t},\dot{z}_t)|_{t=t_0}$ are nonzero. Since this inequality holds for each fixed $t_0 \in [0,1]$, the strict version of $(iii)$ is proved.
		\end{proof}

\begin{corollary}\label{FourthOrder4}
	Assuming \Gzero-\Gtwo, \Gfour-\Geight\ and $G\in C^4(cl(X\times Y \times Z)
	)$,  then \Gthree\ is equivalent to the following statement:
	 For any given curve $x_s\in X$ connecting $x_0$ and $x_1$,  and any curve $(y_t, z_t) \in cl( Y \times Z)$ connecting $(y_0,z_0)$ and $(y_1, z_1)$, we have 
	\begin{equation}\label{foc}
		\frac{\partial^2}{\partial s^2 }\Biggl(\frac{1}{G_z(x_s, y_t, z_t)}\frac{\partial^2}{\partial t^2} G(x_s,y_t,z_t) \Biggr)\Bigg|_{(s,t) = (s_0,t_0)}\le 0,
	\end{equation}
	whenever either of the two curves $t\in [0,1] \longmapsto (G_x, G)(x_{s_0}, y_t, z_t)$ and  $s\in [0,1] \longmapsto \frac{G_y}{G_z}(x_s, y_{t_0}, z_{t_0})$ forms an affinely parametrized line segment.\\
\end{corollary}	



\bigskip

\section{Concavity of principal's objective when her utility does not depend directly on agents' private types:  
A sharper, more local result} 
\label{A:B}

In this section, we reveal a necessary and sufficient condition  for the concavity of principal's maximization problem, not only for the specific example as above, but for many other private-value circumstances, where principal's utility only directly depends on the products sold and their selling prices, but not the buyer's type. \\



\begin{proposition}[Concavity of principal's objective when her payoff is independent of agents' types]\label{bar{G}^*-Concavity2}
	Suppose $G \in C^3(cl(X\times Y \times Z)
	)$ satisfies \Gzero-\Gsix,  $\pi \in C^2( cl( Y \times Z)
	)$, and assume there exists a set $J\subset cl(X)$ such that for each $\bar{y}\in Y\times Z$, $ 0\in ( \pi_{\bar{y}}+G_{\bar{y}})(cl(J), \bar{y})$.  Then the following statements are equivalent:
	\begin{enumerate}[(i)]
		\item \text{\rm local ${\bar G}^*$-concavity of} $\pi$: i.e. $\pi_{\bar{y}\bar{y}}(\bar{y}) + G_{\bar{y}\bar{y}}(x, \bar{y})$ is non-positive definite whenever
 $(x, \bar{y}) \in cl(J)\times Y \times Z$ satisfies $\pi_{\bar{y}}(\bar{y})+ G_{\bar{y}}(x, \bar{y})=0$;
		\item $\pmb \Pi$ is concave on $\mathcal{U}$ for all $\mu\ll \mathcal{L}^m$.
	\end{enumerate}
\end{proposition}

\begin{remark}
	The sufficient condition, i.e., existence of $J \subset cl(X)$(, such that for each  $\bar{y}\in Y\times Z$, $ 0\in ( \pi_{\bar{y}}+G_{\bar{y}})(cl(J), \bar{y})$), make the statement more general than taking some specific subset of $cl(X)$ instead. In particular, if $J=cl(X)$, this condition is equivalent to: for each $\bar{y} \in Y\times Z$, there exists $x \in cl(X)$, such that $( \pi_{\bar{y}}+G_{\bar{y}})(x, \bar{y}) = 0$. One of its economic interpretations is that for each product-price type, there exists a customer type, such that his marginal disutility, the gradient with respect to product type (e.g., quality, quantity, etc.) and price type, coincides with the marginal utility of the monopolist. 
\end{remark}

\begin{proof}[Proof of Proposition \ref{bar{G}^*-Concavity2}]
	 	$(i)\Rightarrow (ii).$ Similar to the proof of Corollary \ref{bar{G}^*-Concavity}, we only need to show non-positive definiteness of $(\pi_{\bar{k}\bar{j}}- \pi_{\bar{l}} \bar{G}^{\bar i,\bar l}\bar{G}_{\bar{i},\bar{k}\bar{j}})$, i.e., for any $\bar{x} = (x,x_0) \in X \times X_0$,  $\bar{y} \in Y\times Z$ and $\xi \in \R^{n+1} $, $\big(\pi_{\bar{k}\bar{j}}(\bar{y})- \pi_{\bar{l}}(\bar{y}) \bar{G}^{\bar i,\bar l}(\bar{x}, \bar{y})\bar{G}_{\bar{i},\bar{k}\bar{j}}(\bar{x}, \bar{y})\big)\xi^{\bar{k}}\xi^{\bar{j}} \le 0$.\\
	 
	 	For any fixed $\bar{x} = (x, x_0) \in X \times X_0$, $\bar{y} \in Y\times Z$, $\xi \in \R^{n+1}$, there exist $\delta >0$ and a curve $t \in (-\delta, \delta) \longmapsto \bar{y}_t \in Y\times Z$, such that $\bar{y}_t|_{t=0} =\bar{y}$,  $\dot{\bar{y}}_t|_{t=0} = \xi$ and $\frac{d^2}{dt^2} \bar{G}_{\bar{i}, }(\bar{x}, \bar{y}_t) = 0$. Thus, 
	 	\begin{equation}\label{EqnSecondDiffbar{G}3}
	 	0 = \frac{d^2}{dt^2}\bigg|_{t=0}\bar{G}_{\bar{i},}(\bar{x}, \bar{y}_t) = \bar{G}_{\bar{i}, \bar{k}\bar{j}}(\bar{x}, \bar{y}) \xi^{\bar{k}}\xi^{\bar{j}} + \bar{G}_{\bar{i}, \bar{k}}(\bar{x}, \bar{y})\cdot (\ddot{\bar{y}}_t)^{\bar{k}}\big|_{t=0}
	 	\end{equation}
	 	For this $\bar{y}$, since $0\in (\pi_{\bar{y}}+G_{\bar{y}})(cl(J), \bar{y})$, there exists $x^*\in cl(J)$, such that $( \pi_{\bar{y}}+G_{\bar{y}})(x^*, \bar{y})=0$. By property $(i)$, one has $(\pi_{\bar{y}\bar{y}}(\bar{y}) + G_{\bar{y}\bar{y}}(x^*, \bar{y}))\xi^{\bar{k}}\xi^{\bar{j}} \le 0$.
	 	Let $\bar{x}^* =(x^*,-1)$, then $\pi_{\bar{l}}(\bar{y}) = \bar{G}_{,\bar{l}}({\bar{x}}^*, \bar{y})$  and $\pi_{\bar{k}\bar{j}}(\bar{y})\xi^{\bar{k}}\xi^{\bar{j}} \le \bar{G}_{,\bar{k}\bar{j}}({\bar{x}^*}, \bar{y})\xi^{\bar{k}}\xi^{\bar{j}}$, for each $\bar{l} = 1, 2,..., n+1$. Thus, combining (\ref{EqnSecondDiffbar{G}3}) and \Gthree, we have 
	 	\begin{flalign}\label{Eqn:proofcor2}
	 	\begin{aligned}
	 	&\big(\pi_{\bar{k}\bar{j}}(\bar{y})- \pi_{\bar{l}}(\bar{y}) \bar{G}^{\bar i,\bar l}(\bar{x}, \bar{y})\bar{G}_{\bar{i},\bar{k}\bar{j}}(\bar{x}, \bar{y})\big)\xi^{\bar{k}}\xi^{\bar{j}} \\
	 	&\le
	 	\big(\bar{G}_{,\bar{k}\bar{j}}({\bar{x}}^*, \bar{y})- \bar{G}_{,\bar{l}}({\bar{x}}^*, \bar{y})\bar{G}^{\bar i,\bar l}(\bar{x}, \bar{y})\bar{G}_{\bar{i},\bar{k}\bar{j}}(\bar{x}, \bar{y})\big)\xi^{\bar{k}}\xi^{\bar{j}} \\
	 	&= \bar{G}_{,\bar{k}\bar{j}}({\bar{x}}^*, \bar{y})\xi^{\bar{k}}\xi^{\bar{j}}+ \bar{G}_{,\bar{l}}({\bar{x}}^*, \bar{y})\cdot (\ddot{\bar{y}}_t)^{\bar{l}}\big|_{t=0}\\
	 	& = \frac{d^2}{dt^2}\bigg|_{t=0} \bar{G}({\bar{x}}^*, \bar{y}_t)\\
	 	& = - \frac{d^2}{dt^2}\bigg|_{t=0} G({x}^*, \bar{y}_t)\\
	 	&\le 0.
	 	\end{aligned}
	 	\end{flalign}

	    $(ii)\Rightarrow (i).$  
	 	For any $(x,\bar{y}) \in cl(J)\times {Y} \times {Z}$, satisfying $\pi_{\bar{y}}(\bar{y}) +  G_{\bar{y}}(x, \bar{y})=0$, we would like to show  $(\pi_{\bar{k}\bar{j}}(\bar{y}) + {G}_{,\bar{k}\bar{j}}({x}, \bar{y}))\xi^{\bar{k}}\xi^{\bar{j}} \le 0$, for any $\xi\in \R^{n+1}$. Let $\bar{x} = (x, -1)$, there exist $\delta >0$ and a curve $t \in (-\delta, \delta) \longmapsto \bar{y}_t \in Y\times Z$, such that $\bar{y}_t|_{t=0} =\bar{y}$,  $\dot{\bar{y}}_t|_{t=0} = \xi$ and $\frac{d^2}{dt^2} \bar{G}_{\bar{i}, }(\bar{x}, \bar{y}_t) = 0$. Thus, equation \eqref{EqnSecondDiffbar{G}3} holds.

	 	Since $\pmb \Pi$ is concave, by Theorem $\ref{maintheorem}$ and Lemma  $\ref{LemmaProfitConcavity}$ as well as equation \eqref{EqnSecondDiffbar{G}3}, we have 
	 	\begin{flalign*}
	 		0 \ge& \big(\pi_{\bar{k}\bar{j}}(\bar{y})- \pi_{\bar{l}}(\bar{y}) \bar{G}^{\bar i,\bar l}(\bar{x}, \bar{y})\bar{G}_{\bar{i},\bar{k}\bar{j}}(\bar{x}, \bar{y})\big)\xi^{\bar{k}}\xi^{\bar{j}} \\
	 		=& \big(\pi_{\bar{k}\bar{j}}(\bar{y}) - \bar{G}_{,\bar{k}\bar{j}}({\bar{x}}, \bar{y}) + \bar{G}_{,\bar{k}\bar{j}}({\bar{x}}, \bar{y})-  \bar{G}_{,\bar{l}}({\bar{x}}, \bar{y})\bar{G}^{\bar i,\bar l}(\bar{x}, \bar{y})\bar{G}_{\bar{i},\bar{k}\bar{j}}(\bar{x}, \bar{y})\big)\xi^{\bar{k}}\xi^{\bar{j}} \\
	 		=& \big(\pi_{\bar{k}\bar{j}}(\bar{y}) - \bar{G}_{,\bar{k}\bar{j}}({\bar{x}}, \bar{y})\big)\xi^{\bar{k}}\xi^{\bar{j}} + \frac{d^2}{dt^2}\bigg|_{t=0} \bar{G}(\bar{x}, \bar{y}_t)\\
	 		=&\big(\pi_{\bar{k}\bar{j}}(\bar{y}) - \bar{G}_{,\bar{k}\bar{j}}({\bar{x}}, \bar{y})\big)\xi^{\bar{k}}\xi^{\bar{j}} \\
	 		=& (\pi_{\bar{k}\bar{j}}(\bar{y}) + {G}_{,\bar{k}\bar{j}}({x}, \bar{y}))\xi^{\bar{k}}\xi^{\bar{j}},
	 	\end{flalign*}
	 	which completes the proof.
\end{proof}

The following remark provides an equivalent condition for  uniformly concavity of principal's maximization problem. Its proof is very similar to that of the above proposition.

\begin{remark}\label{R:B2}
 In addition to the hypotheses of Proposition \ref{bar{G}^*-Concavity2}, when
 $\bar{z}=+\infty$ assume the homeomorphisms of \Gone\ are uniformly bi-Lipschitz. Then
 the following statements are equivalent:
	\begin{enumerate}[(i)]
		\item $\pi_{\bar{y}\bar{y}}(\bar{y}) + G_{\bar{y}\bar{y}}(x, \bar{y})$ is  uniformly negative definite 
for all $(x, \bar{y}) \in cl(J)\times Y \times Z$ such that $\pi_{\bar{y}}(\bar{y}) + G_{\bar{y}}(x, \bar{y})=0$;
		\item $\pmb \Pi$ is uniformly concave  on $\mathcal{U} \subset W^{1,2}(X,d\mu)$, uniformly for all $\mu\ll \mathcal{L}^m$.
	\end{enumerate}
\end{remark}

	When $m=n$, $G(x,y,z)= b(x,y)-z \in C^3(cl(X\times Y \times Z)
	)$ satisfies \Gzero-\Geight,  and $\pi(y,z)=z-a(y) \in C^2( cl( Y \times Z)
	)$, then Corollary \ref{bar{G}^*-Concavity} shows $b^*$-convexity of $a$ is a sufficient condition for
	concavity of $\pmb \Pi$  for all $\mu\ll \mathcal{L}^m$. One may wonder under what hypotheses it would become a necessary condition as well. From Theorem A.1 in \cite{KimMcCann10}, under the same assumptions as above, the manufacturing cost $a$ is $b^*$-convex if and only if it satisfies the following local $b^*$-convexity hypothesis: $D^2a(y)\ge D^2_{yy}b(x,y)$ whenever $Da(y) = D_{y}b(x,y)$. 
	Combined with Proposition \ref{bar{G}^*-Concavity2}, we have the following corollary.

\begin{corollary}
Adopting the terminology of [FKM11],  i.e. (B0)-(B4), $G(x,y,z) = b(x,y)- z \in C^3(cl(X\times Y \times Z)$ and 
$\pi(x,y,z) = z -a(y) \in C^2( cl( Y \times Z)$, then $a(y)$ is $b^*$-convex if  and only if $\mathbf \Pi$ is concave on $\mathcal U$ and for every $y \in Y$, there exists $x \in cl(X)$ such that $Da(y) = D_y b(x,y)$.


\end{corollary}

\begin{proof}
	Assume $a$ is $b^*$-convex, by definition, there exists a function  $a^*: cl(X) \rightarrow \R$, such that for any $y\in Y$, $a(y) = \max_{x\in cl(X)} b(x,y) - a^*(x)$. Therefore, for any $y_0 \in Y$, there exists $x_0 \in cl(X)$, such that $a(y) \ge b(x_0,y) - a^*(x_0)$ for all $y \in Y$, with equality holds at $y = y_0$. This implies, $Da(y_0) = D_y b(x_0, y_0)$. Taking $J=cl(X)$ and applying Proposition \ref{bar{G}^*-Concavity2}, we have concavity of $\mathbf \Pi$, since local $b^*$-convexity of $a$ is automatically satisfied by a $b^*$-convex function $a$. \\
	On the other hand, assuming $\mathbf \Pi$ is concave on $\mathcal U$ and for every $y \in Y$, there exists $x \in cl(X)$ such that $Da(y) = D_y b(x,y)$, Proposition \ref{bar{G}^*-Concavity2} implies local $b^*$-convexity of $a$. Together with Theorem A.1 in \cite{KimMcCann10}, we know $a$ is $b^*$-convex.
\end{proof}

%

\bigskip

\end{document}